\documentclass[a4paper, 11pt, twoside, reqno]{amsart}

\usepackage{amsmath,amsfonts,amsthm,amsopn,xcolor,amssymb,enumitem}
\usepackage[colorlinks=true,urlcolor=blue,citecolor=red,linkcolor=blue,linktocpage,pdfpagelabels,bookmarksnumbered,bookmarksopen]{hyperref}
\hypersetup{urlcolor=blue, citecolor=blue, linkcolor=blue}

\usepackage[left=2.61cm,right=2.61cm,top=2.72cm,bottom=2.72cm]{geometry}

\usepackage{color,epsfig,graphics}
\usepackage{latexsym}
\usepackage{graphicx}
\usepackage{mathrsfs}
\usepackage{bm}

\usepackage[T1]{fontenc}

\numberwithin{equation}{section}

\newtheorem{theorem}{Theorem}[section]
\newtheorem{lemma}[theorem]{Lemma}
\newtheorem{remark}[theorem]{Remark}
\newtheorem{proposition}[theorem]{Proposition}

\newtheorem*{claim*}{Claim}

\newcommand{\ef}{\eqref}
\newcommand{\dis}{\displaystyle}

\newcommand{\C}{\mathbb{C}}

\newcommand{\R}{\mathbb{R}}

\newcommand{\ep}{\varepsilon}
\newcommand{\la}{\lambda}

\newcommand{\RN}{{\mathbb{R}^N}}
\newcommand{\intR}{\int_{\mathbb{R}^3}}

\newcommand{\HT}{H^1(\R^3, \C)}

\newcommand{\CA}{\mathcal{A}}
\newcommand{\CI}{\mathcal{I}}

\newcommand{\CM}{\mathcal{M}}
\newcommand{\CS}{\mathcal{S}}

\newcommand{\IM}{\,\mathrm{Im}}
\newcommand{\RE}{\,\mathrm{Re}}
\newcommand{\Div}{\,\mathrm{div}}

\title
[Schr\"odinger-Poisson system with a doping profile]
{Strong instability of standing waves 
for $L^2$-supercritical Schr\"odinger-Poisson system with a doping profile}

\author[M.~Colin]{Mathieu Colin}
\author[T.~Watanabe]{Tatsuya Watanabe}

\address[M.~Colin]{\newline\indent
University of Bordeaux, CNRS, Bordeaux INP, 
IMB, UMR 5251,  F-33400, Talence, France
\newline\indent
 IMB, UMR 5251,  F-33400, Talence, France
}
\email{mathieu.colin@math.u-bordeaux.fr}

\address[T.~Watanabe]{\newline\indent 
Department of Mathematics, Faculty of Science, Kyoto Sangyo University,
\newline\indent
Motoyama, Kamigamo, Kita-ku, Kyoto-City, 603-8555, Japan}
\email{tatsuw@cc.kyoto-su.ac.jp}

\thanks{}

\date{\today}

\begin{document}

\begin{abstract}
This paper is concerned with 
the nonlinear Schr\"odinger-Poisson system with a doping profile.
We are interested in the strong instability of standing waves 
associated with ground state solutions in the $L^2$-supercritical case.
The presence of a doping profile causes several difficulties, especially 
in examining geometric shapes of fibering maps
along an $L^2$-invariant scaling curve.
Furthermore, the classical approach by Berestycki-Cazenave for the strong instability
cannot be applied to our problem due to a remainder term caused by the doping profile.
To overcome these difficulties,
we establish a new energy inequality associated with the $L^2$-invariant scaling
and adopt the strong instability result developed in \cite{FO1}.
When the doping profile is a characteristic function supported 
on a bounded smooth domain, some geometric quantities
related to the domain, such as the mean curvature,
are responsible for the strong instability of standing waves.
\end{abstract}

\maketitle

\section{Introduction}

In this paper, we are concerned with the following nonlinear 
Schr\"odinger-Poisson system:
\begin{equation} \label{eq:1.1}
\begin{cases}
& -\Delta u+\omega u+e \phi u=|u|^{p-1} u \\
& -\Delta \phi =\frac{e}{2} \left( |u|^2- \rho(x) \right)
\end{cases} 
\quad \hbox{in} \ \R^3,
\end{equation}
where $\omega >0 $, $e>0$.
Equation \ef{eq:1.1} appears as a stationary problem for the 
time-dependent nonlinear Schr\"odinger-Poisson system:
\begin{equation} \label{eq:1.2}
\begin{cases}
i \psi_t + \Delta \psi - e\phi \psi + |\psi|^{p-1} \psi =0 
\quad \hbox{in} \ \R_+ \times \R^3, \\
-\Delta \phi = \frac{e}{2}\big( |\psi|^2-\rho(x) \big) 
\quad \hbox{in} \ \R_+ \times \R^3, \\
\psi(0,x)=\psi_0 \quad \hbox{in } \R^3.
\end{cases}
\end{equation}
Indeed when we look for a standing wave of the form: $\psi(t,x)=e^{i \omega t} u(x)$,
we are led to the elliptic problem \ef{eq:1.1}.
Here, we are interested in the strong instability of standing waves 
corresponding to ground state solutions of \ef{eq:1.1} when $\frac{7}{3} < p < 5$.

The Schr\"odinger-Poisson system appears 
in various fields of physics, such as quantum mechanics,
black holes in gravitation and plasma physics.
Especially, this system plays an important role
in the study of semi-conductor theory; see \cite{Je, MRS, Sel},
and then the function $\rho(x)$ is referred as \textit{impurities} 
or a \textit{doping profile}.
The doping profile comes from 
the difference of density numbers of positively charged donor ions
and negatively charged acceptor ions,
and the most typical examples are characteristic functions, step functions 
or Gaussian functions.
Equation \ef{eq:1.1} also appears as a stationary problem
for the Maxwell-Schr\"odinger system.
We refer to \cite{BF, CW, CW2} for the physical background
and the stability result of standing waves for the Maxwell-Schr\"odinger system.
In this case, the constant $e$ describes the strength of the interaction
between a particle and an external electromagnetic field.

On one hand, the nonlinear Schr\"odinger-Poisson system with $\rho \equiv 0$:
\begin{equation} \label{eq:1.3}
\begin{cases}
& -\Delta u+\omega u+e \phi u=|u|^{p-1} u \\
& -\Delta \phi =\frac{e}{2} |u|^2
\end{cases} 
\quad \hbox{in} \ \R^3
\end{equation}
has been studied widely in the last two decades.
Especially, the existence of non-trivial solutions and 
ground state solutions of \ef{eq:1.3}
has been considered in detail.
Furthermore, the existence of associated $L^2$-constraint minimizers 
depending on $p$ and the size of the mass
and their stability have been investigated as well.
We refer to e.g. \cite{AP, BJL, BS1, CDSS, CW, K2, LZH, R, TC, WL1, WL2, ZZ}
and references therein.
The strong instability of standing waves associated with ground state solutions of
\ef{eq:1.3} has been established in \cite{BJL} for the $L^2$-supercritical case
and in \cite{FCW} for the $L^2$-critical case respectively.

On the other hand, the nonlinear Schr\"odinger-Poisson system with
a doping profile is less studied.
In \cite{DeLeo, DR}, the corresponding 1D problem has been considered.
Moreover, the linear Schr\"odinger-Poisson system 
(that is, the problem \ef{eq:1.1} without $|u|^{p-1}u$)
with a doping profile in $\R^3$ has been studied in \cite{Ben1, Ben2}.
In \cite{CW4}, the authors have investigated the existence of 
stable standing waves for \ef{eq:1.2} by considering the corresponding
$L^2$-minimization problem in the case $2<p< \frac{7}{3}$.
The existence of ground state solutions for \ef{eq:1.1}
in the case $2<p<5$ has been obtained in \cite{CW5}. 
Moreover in \cite{CW5}, 
the authors have established the relation between
ground state solutions of \ef{eq:1.1} 
and $L^2$-constraint minimizers obtained in \cite{CW4}.
But as far as we know, there is no literature 
concerning with the strong instability of standing waves 
associated with ground state solutions of \ef{eq:1.1} in the $L^2$-supercritical case,
which is exactly the purpose of this paper.

\smallskip
To state our main results, let us give some notations. 
For $u \in H^1(\R^3,\C)$, 
the energy functional associated with \ef{eq:1.1} is given by
\begin{align}
\mathcal{I}(u) &= \frac{1}{2} \intR |\nabla u |^2 \,dx + \frac{\omega}{2} \intR |u|^2 \,dx
-\frac{1}{p+1} \intR |u|^{p+1} \,d x + e^2 \CA(u). \label{eq:1.4}
\end{align}
Here we denote the nonlocal term by $S(u)=S_0(u)+S_1$ with
\begin{align}
S_0(u)(x)&:=(-\Delta)^{-1} \left( \frac{|u(x)|^2}{2} \right) 
=\frac{1}{8 \pi |x|} *|u(x)|^2, \label{nonlocal1} \\
S_1(x)&:= (-\Delta)^{-1} \left( \frac{- \rho(x)}{2} \right) 
= -\frac{1}{8 \pi |x|} * \rho(x),  \label{nonlocal2}
\end{align}
and the functional corresponding to the nonlocal term by
\[
\begin{aligned}
\CA(u) 
&:= \frac{1}{4} \intR S(u) \big( |u|^2-\rho(x) \big) \,dx \\
&= \frac{1}{32\pi} \intR \intR
\frac{\big( |u(x)|^2-\rho(x) \big) \big( |u(y)|^2-\rho(y) \big)}{|x-y|} \,dx\,dy.
\end{aligned}
\]
A function $u_0$ is said to be a \textit{ground state solution} (GSS) 
of \ef{eq:1.1}
if $u_0$ has a least energy among all nontrivial solutions of \ef{eq:1.1},
namely $u_0$ satisfies
\[
\mathcal{I}(u_0) = \inf \{ \mathcal{I}(u) \mid u \in \HT
\setminus \{0 \}, \ \mathcal{I}'(u)=0 \}.
\]
For the doping profile $\rho$, we assume that
\begin{equation} \tag{A1} \label{A1}
\rho(x) \in L^{\frac{6}{5}}(\R^3) \cap L^q_{loc}(\R^3) \ \text{for some} \ q>3, 
\ x \cdot \nabla \rho(x) \in L^{\frac{6}{5}}(\R^3),
\ x \cdot (D^2 \rho(x) x) \in L^{\frac{6}{5}}(\R^3),
\end{equation}
where $D^2 \rho$ is the Hessian matrix of $\rho$, and 
\begin{equation} \tag{A2} \label{A2}
\rho(x) \ge 0, \ \not\equiv 0 \quad \text{for} \ x \in \R^3.
\end{equation}
Typical examples are the Gaussian function $\rho(x) =\ep e^{-\alpha |x|^2}$
and $\rho(x) =\frac{\ep}{1+|x|^{\alpha}}$ for $\alpha> \frac{5}{2}$.
In this setting, the following result is known; see \cite{CW5}.

\begin{proposition} \label{prop:1.1}
Suppose that $2<p<5$ and assume that {\rm(A1)}-{\rm(A2)} are satisfied.
There exists $\rho_0$ independent of $e$, $\rho$ such that if
\[
e^2 \left(\|\rho\|_{L^{\frac{6}{5}}(\R^3)}
+\|x \cdot \nabla \rho\|_{L^{\frac{6}{5}}(\R^3)}
+\| x \cdot (D^2 \rho x ) \|_{L^{\frac{6}{5}}(\R^3)} \right) \le \rho_0,
\]
then \ef{eq:1.1} has a ground state solution $u_0$.
Moreover any ground state solution of \ef{eq:1.1} is real-valued
up to phase shift.
\end{proposition}

It was shown in \cite{CW5} that the ground state solution $u_0$ is 
characterized as a minimizer of $\mathcal{I}$
on a \textit{Nehari-Pohozaev} set 
$\{ J(u)=0 \}$, 
where $J(u)=2N(u) - P(u)$, 
$N(u)=0$ is the Nehari identity
and $P(u)=0$ is the Pohozaev identity. See also Section 4 below.

\smallskip
Next we proceed to the instability result.
As mentioned later, 
the Cauchy problem for \ef{eq:1.2} is locally well-posed 
in the energy space $H^1(\R^3,\C)$ 
and the maximal existence time $T^*=T^*(\psi_0)$ corresponding to 
$\psi_0 \in H^1(\R^3,\C)$ is well defined.
We say that the (unique) solution $\psi(t,x)$ of \ef{eq:1.2} blows up in finite time
if $T^* < \infty$.
Furthermore, the standing wave $e^{i \omega t}u(x)$ of \ef{eq:1.2}
is said to be \textit{strongly unstable}
if for any $\ep>0$, there exists $\psi_0 \in H^1(\R^3, \C)$
such that $\| \psi_0 - u \|_{H^1(\R^3)} < \ep$
but the solution $\psi(t,x)$ of \ef{eq:1.2} with $\psi(0,x)=\psi_0$
blows up in finite time.

In order to prove the strong instability of standing waves 
corresponding to ground state solutions of \ef{eq:1.1},
we further impose the following condition on the doping profile:
\begin{equation} \tag{A3} \label{A3}
\text{There exist $\alpha>2$ and $C>0$ such that} \ 
\rho(x) \le \frac{C}{1+|x|^{\alpha}} \quad \text{for} \ x \in \R^3.
\end{equation}
We will see that (A3) guarantees 
the exponential decay of ground state solutions of \ef{eq:1.1} at infinity.
Under these preparations,
we are able to state and show the following result.

\begin{theorem} \label{thm:1.1}
Suppose that $\frac{7}{3}<p<5$ and assume that {\rm(A1)}-{\rm(A3)} hold.
There exists $\rho_0$ independent of $e$, $\rho$ such that if
\[
e^2 \left(\|\rho\|_{L^{\frac{6}{5}}(\R^3)}
+\|x \cdot \nabla \rho\|_{L^{\frac{6}{5}}(\R^3)}
+\| x \cdot (D^2 \rho x ) \|_{L^{\frac{6}{5}}(\R^3)} \right) \le \rho_0,
\]
then the standing wave $e^{i \omega t} u_0$ is strongly unstable.
\end{theorem}

\smallskip
We emphasize that no restriction on the frequency $\omega$ is required
in Theorem \ref{thm:1.1}.
The assumption (A1) rules out the case 
$\rho$ is a characteristic function supported on a bounded smooth domain.
Even in this case, we are still able to obtain the strong instability of 
ground state solutions
under a smallness condition on some geometric quantities related to the domain;
See Section 7.
We also mention that our instability result heavily relies on the fact 
$p$ is $L^2$-supercritical.
For the moment, we don't know whether the strong instability holds for $p= \frac{7}{3}$.

\smallskip
Here we briefly explain our strategy and its difficulty.
It is known that the strong instability of standing waves
is based on the \textit{virial identity}:
\[
V''(t) = 8 Q \big ( \psi(t) \big).
\]
Here $V(t)$ is the variance defined in \ef{var} below, while
the functional $Q$ is defined by 
\begin{align} \label{energy}
Q(u) &:= \| \nabla u \|_{L^2(\R^3)}^2 
- \frac{3(p-1)}{2(p+1)} \| u \|_{L^{p+1}(\R^3)}^{p+1} \notag \\
&\quad -\frac{e^2}{2} \int_{\R^3} S_0(u)|u|^2 \,dx
+2 e^2 \intR S_1 |u|^2 \,dx 
-e^2 \int_{\R^3} S_2 |u|^2 \,d x, \\
S_2(x) &:= (-\Delta )^{-1} \left( \frac{x \cdot \nabla \rho(x)}{2} \right)
= \frac{1}{8 \pi |x|} * \big( x \cdot \nabla \rho(x) \big), \label{nonlocal3}
\end{align}
and $Q(u)$ can be obtained formally by considering 
$\frac{d}{d\la} \mathcal{I}(u^{\la})|_{\la =1} =0$ for 
the $L^2$-invariant scaling:
\begin{equation} \label{invariant}
u^{\la}(x)= \la^{\frac{3}{2}} u(\la x).
\end{equation}
Then we are able to prove the strong instability of standing waves for $u_0$
if we could show that there exists a subset $\mathcal{B} \subset H^1(\R^3,\C)$ 
such that 
\begin{equation} \label{key1}
\frac{1}{2} Q(u) \le \mathcal{I}(u) - \mathcal{I}(u_0) < 0
\quad \text{for all} \ u \in \mathcal{B}
\end{equation}
and
\begin{equation} \label{key2}
u_0^{\la} \in \mathcal{B} \quad \text{for all} \ \la>1.
\end{equation}
In fact, \ef{key1} together with the conservation laws
shows that the set $\mathcal{B}$ is invariant under the flow for Equation \ef{eq:1.2}
and $V''(t)<0$ for $\psi(0,x)=\psi_0 \in \mathcal{B}$,
which implies that $\psi(t,x)$ blows up in finite time.
Moreover by \ef{key2}, there exists $\psi_0$ satisfying 
$\| \psi_0 - u_0 \|_{H^1(\R^3)} \sim 0$ and $\psi_0 \in \mathcal{B}$,
concluding that $e^{i\omega t} u_0$ is strongly unstable.
In \cite{BJL, FCW} where the case $\rho \equiv 0$ was treated, 
it was established that the ground state solution $u_0$ of \ef{eq:1.3} 
has the variational characterization:
\begin{equation} \label{char}
\mathcal{I}(u_0) = \inf\{ \mathcal{I}(u) \mid u \in H^1(\R^3,\C) \setminus \{ 0 \}, 
Q(u)=0 \}.
\end{equation}
Then the proof of \ef{key1} and \ef{key2}
was carried out by considering  
\[
\mathcal{B} = \{ u \in H^1(\R^3, \C) \mid
\mathcal{I}(u) < \mathcal{I}(u_0), \  Q(u) <0 \},
\]
which is exactly the classical approach by \cite{BC} (see also \cite{LeC}).

As we can easily imagine, if a doping profile $\rho$ is considered, 
scaling arguments do not work straightforwardly 
because of the loss of spatial homogeneity. 
Furthermore the presence of the doping profile $\rho$ 
satisfying (A1)-(A3) causes additional difficulties.
Especially we cannot expect to apply 
the classical approach due to \cite{BJL, BC, FCW, LeC}.
See the discussion in the end of Section 5 below.
To overcome these difficulties, 
we first prove the following energy inequality:
\begin{align} \label{estimate}
&\mathcal{I}(u^{\la}) - \frac{\la^2}{2} Q(u)
- \mathcal{I}(u) + \frac{1}{2}Q(u) \\
&\le - C_1 (1-\la)^2 \| u \|_{L^{p+1}(\R^3)}^{p+1} 
+ C_2(1-\la)^2 e^2 \left( \| \rho \|_{\frac{6}{5}} + \| x \cdot \nabla \rho \|_{\frac{6}{5}} 
+ \| x \cdot (D^2 \rho x) \|_{\frac{6}{5}} \right) \| u \|_{H^1(\R^3)}^2 \notag
\end{align}
for any $u \in H^1(\R^3,\C)$, $\la \le 1$ and 
some $C_1$, $C_2>0$ independent of $e$, $\rho$, $\la$.
The energy estimate \ef{estimate} enables us to prove \ef{key1} for 
$u \in H^1(\R^3,\C)$ satisfying $Q(u) \le 0$ and $J(u) \le 0$.
By setting
\[
\mathcal{B} := \{ u \in H^1(\R^3, \C) \mid
\mathcal{I}(u) < \mathcal{I}(u_0), \ J(u) < 0 , \ Q(u) <0 \},
\]
in our problem, we are able to show that the set $\mathcal{B}$ is 
invariant under the flow of \ef{eq:1.2}.
Furthermore in order to obtain \ef{key2}, 
we adopt a strategy developed in \cite{FO1}
(see also \cite{FO2, O1, O2}).
By establishing $\frac{d^2}{d \la^2} \mathcal{I}(u^{\la})|_{\la =1}<0$,
we can obtain \ef{key2} without using the variational characterization \ef{char}.

When $\rho$ is a characteristic function, 
further consideration is required
because $\rho$ cannot be weakly differentiable.
In this case, a key of the proof is the \textit{sharp boundary trace inequality}
which was developed in \cite{A},
and a variation of domain
related with the \textit{calculus of moving surfaces} 
due to Hadamard \cite{Gri}.
Then by imposing a smallness condition of some geometric quantities
related to the support of $\rho$, 
we are able to obtain the strong instability of standing waves
associated with ground state solutions of \ef{eq:1.1}.

\smallskip
This paper is organized as follows.
In Section 2, we recall a known result on the Cauchy problem for \ef{eq:1.2}
and establish the virial identity associated with \ef{eq:1.2}.
In Section 3, we prepare several properties of the energy functional
and some lemmas which will be used later on.
We introduce fundamental properties of 
ground state solutions of \ef{eq:1.1} in Section 4.
In Section 5, we investigate several fibering maps
along the $L^2$-invariant scaling \ef{invariant}.
We complete the proof of Theorem \ref{thm:1.1} in Section 6. 
In Section 7, we finish this paper   
by considering the case where $\rho$ is a characteristic function
and present the strong instability of standing waves for this case.

Hereafter in this paper, 
unless otherwise specified, we write $\| u\|_{L^{p}(\R^3)}=\| u\|_p$.
We also set $\| u \|^2 := \| \nabla u \|_2^2 + \| u \|_2^2$.

\section{Well-posedness of the Cauchy problem and the virial identity}

In this section, we consider the Cauchy problem:
\begin{equation} \label{eq:2.1}
\begin{cases}
i \psi_t + \Delta \psi - e^2 S(\psi) \psi + |\psi|^{p-1} \psi =0 
\quad \hbox{in} \ \R_+ \times \R^3, \\
\psi(0,x)=\psi_0 \quad \hbox{in } \R^3,
\end{cases}
\end{equation}
where $e>0$, $1<p<5$, 
$S(\psi)= \frac{1}{2} (-\Delta)^{-1} ( |\psi|^2-\rho)$
and $\psi_0 \in H^1(\R^3,\C)$.
Then we have the following result on the local well-posedness.

\begin{proposition} \label{prop:2.1}
Assume $\rho \in L^{\frac{6}{5}}(\R^3)$.
There exists $T^*=T^*(\| \psi_0\|_{H^1})>0$ such that 
\ef{eq:2.1} has a unique solution $\psi \in X$, where
\[
X= \left\{ \psi \in C\left( [0,T^*],H^1(\R^3, \C) \right) 
\cap L^{\infty}\left( (0,T^*), H^1(\R^3, \C) \right) \right\}.
\]
Furthermore, $\psi$ satisfies the energy conservation law 
and the charge conservation law:
\[
\mathcal{E}\big( \psi(t) \big) = \mathcal{E}(\psi_0)
\quad \hbox{and} \quad 
\| \psi(t) \|_2 = \| \psi_0 \|_2 \quad \hbox{for all} \ t \in (0,T^*).
\]
\end{proposition}

The proof of Proposition \ref{prop:2.1} can be found in \cite{CW4}.
(See also \cite[Proposition 3.2.9, Theorem 4.3.1 and Corollary 4.3.3]{Ca}.)

Next we establish the virial identity for \ef{eq:2.1},
which plays a fundamental role in the study of the strong instability.
Let us define the weighted function space:
\[
\Sigma :=\left\{ u \in H^1(\mathbb{R}^3, \mathbb{C}) \mid
\int_{\mathbb{R}^3} |x|^2\ | u|^2 \,d x<\infty \right\}.
\]
Then we are able to show the following.

\begin{lemma} \label{lem:2.2}
Assume that $\psi_0 \in \Sigma$ and 
let $T^*>0$ be the maximal existence time associated with $\psi_0$.
For the unique solution $\psi$ of \ef{eq:2.1}, 
let us denote by $V(t)$ the variance:
\begin{equation} \label{var}
V(t) := \int_{\R^3} |x|^2| \psi(t,x)|^2 \,d x.
\end{equation}
Then the following identity holds:
\begin{equation} \label{eq:2.2}
V^{\prime \prime}(t)= 8 Q \big( \psi(t) \big) 
\quad \text {for all} \ t \in [0, T^*),
\end{equation}
where $Q$ is the functional defined in \ef{energy}.
\end{lemma}

\begin{proof}
First we show that
\begin{equation} \label{eq:2.3}
V^{\prime}(t) =4 \IM \intR (x \cdot \nabla \psi) \bar{\psi} \,d x.
\end{equation}
Indeed multiplying \ef{eq:2.1} by $2\bar{\psi}$ and taking the imaginary part, 
one finds that
\begin{equation} \label{eq:2.4}
\frac{\partial}{\partial t} |\psi|^2 
= \IM \left( i \frac{\partial}{\partial t} |\psi|^2\right)
= \IM ( 2i \psi_t \bar{\psi})  
= -2 \IM (\bar{\psi} \Delta \psi)
= -2 \Div \big( \IM ( \bar{\psi} \nabla \psi) \big). 
\end{equation}
Moreover multiplying \ef{eq:2.4} by $|x|^2$ and integrating over $\R^3$, we get
\[
\begin{aligned}
\frac{\partial}{\partial t} \int_{\mathbb{R}^3} |x|^2\ | \psi|^2 \,d x
&= -2 \int_{\R^3} |x|^2 \Div \big( \IM (\bar{\psi} \nabla \psi) \big) \,d x \\
&= -2 \int_{\mathbb{R}^3} \Div \left( |x|^2 \IM(\bar{\psi} \nabla \psi )\right) \,dx
+2 \int_{\mathbb{R}^3} \nabla |x|^2 \cdot \IM (\bar{\psi} \nabla \psi) \,dx \\
&= 4 \IM \intR (x \cdot \nabla \psi) \bar{\psi} \,d x,
\end{aligned}
\]
yielding that \ef{eq:2.3} holds.

Next we multiply \ef{eq:2.1} by $2x \cdot \nabla \bar{\psi}$, 
integrate over $\R^3$ and take the real part.
Then one has
\begin{align} \label{eq:2.5}
0 &= 2 \RE \int_{\mathbb{R}^3} i(x \cdot \nabla \bar{\psi}) \psi_t \,d x
+2 \RE \int_{\mathbb{R}^3} (x \cdot \nabla \bar{\psi}) \Delta \psi \,d x \notag \\
&\quad -2 e^2 \RE \int_{\mathbb{R}^3} (x \cdot \nabla \bar{\psi}) S(\psi) \psi \,d x
+2 \RE \int_{\mathbb{R}^3} (x \cdot \nabla \bar{\psi}) | \psi|^{p-1} \psi \,d x \notag \\
&=: {\rm I} + {\rm II} + {\rm III} + {\rm IV}. 
\end{align}
It is standard to see that
\[
\begin{aligned}
{\rm I} &= \RE \int_{\mathbb{R}^3} i x \cdot
\left( (\psi \nabla \bar{\psi})_t - \nabla (\psi \bar{\psi}_t ) \right) \,d x 
= \frac{d}{d t} \RE \int_{\mathbb{R}^3} i(x \cdot \nabla \bar{\psi}) \psi \,d x
+3 \RE \int_{\mathbb{R}^3} i \psi \bar{\psi}_t \,d x .
\end{aligned}
\]
Thus from \ef{eq:2.1} and \ef{eq:2.3}, it follows that
\begin{align} \label{eq:2.6}
{\rm I} &= \frac{d}{d t} \IM \int_{\R^3} (x \cdot \nabla \psi) \bar{\psi} \,d x 
+3 \RE \int_{\R^3} \psi 
\left( \Delta \bar{\psi} -e^2 S(\psi) \bar{\psi} +|\psi|^{p-1} \bar{\psi}\right) \,d x 
\notag \\
&= \frac{1}{4} V^{\prime \prime}(t)
-3 \int_{\R^3} |\nabla \psi|^2 \,dx
-3 e^2 \int_{\R^3} S(\psi)|\psi|^2 \,dx 
+3 \int_{\mathbb{R}^3} |\psi|^{p+1} \,dx.
\end{align}
Using the integration by parts, we also have
\begin{align}
{\rm II} &= \int_{\R^3} |\nabla \psi|^2 \,dx, \label{eq:2.7} \\
{\rm IV} &= -\frac{6}{p+1} \int_{\R^3} |\psi|^{p+1} \,dx. \label{eq:2.8}
\end{align}
(See also \cite{Ca}.)

Finally we estimate ${\rm III}$. 
First by the integration by parts, we observe that
\begin{align} \label{eq:2.9}
{\rm  III} &= -e^2 \RE \int_{\mathbb{R}^3} x \cdot \nabla |\psi|^2 S(\psi) \,dx \notag \\
&= -e^2 \RE \int_{\R^3} \Div \left( x S(\psi) |\psi|^2 \right) \,d x 
+e^2 \RE \int_{\R^3} \Div x S(\psi) |\psi|^2 \,dx 
+e^2 \RE \int_{\R^3} x \cdot \nabla S(\psi) |\psi|^2 \,dx \notag \\
&= 3 e^2 \int_{\R^3} S(\psi)|\psi|^2 \,d x
+e^2 \int_{\R^3} x \cdot \nabla S_0(\psi)| \psi|^2 \,dx
+e^2 \int_{\R^3} x \cdot \nabla S_1 |\psi|^2 \,d x.
\end{align}
Moreover by the definition of $S_0$ given in \ef{nonlocal1}
and the Fubini theorem, one finds that
\[
\begin{aligned}
&\int_{\R^3} x \cdot \nabla S_0(\psi)|\psi|^2 \,d x
= \frac{1}{8 \pi} \sum_{j=1}^3 \int_{\mathbb{R}^3} x_j \frac{\partial}{\partial x_j}
\left( \int_{\mathbb{R}^3} \frac{|\psi(y)|^2}{|x-y|} \,d y \right) |\psi(x)|^2 \,d x \\
&= -\frac{1}{8 \pi} \sum_{j=1}^3 \int_{\R^3} \int_{\R^3} 
\frac{x_j (x_j-y_j)}{|x-y|^3} |\psi(y)|^2 |\psi(x)|^2 \,d y \,d x \\
&= -\frac{1}{8 \pi} \int_{\mathbb{R}^3} \int_{\R^3} 
\frac{|\psi(y)|^2 |\psi(x)|^2}{|x-y|} \,dy \,dx 
-\frac{1}{8\pi} \sum_{j=1}^3 \int_{\R^3} \int_{\R^3} 
\frac{y_j(x_j-y_j)}{|x-y|^3} | \psi(y)|^ 2 |\psi(x)|^2 \,dy \,dx \\
&= - \int_{\mathbb{R}^3} S_0(\psi)|\psi|^2 \,d x
-\frac{1}{8 \pi} \sum_{j=1}^3 \int_{\mathbb{R}^3} 
y_j \frac{\partial}{\partial y_j}
\left( \int_{\mathbb{R}^3} \frac{|\psi(x)|^2}{|x-y|} \,d x \right) |\psi(y)|^2 \,dy \\
&= -\int_{\mathbb{R}^3} S_0(\psi) |\psi|^2 \,d x
-\int_{\mathbb{R}^3} y \cdot \nabla S_0(\psi)|\psi|^2 \,d y,
\end{aligned}
\]
from which we deduce that
\begin{equation} \label{eq:2.10}
\int_{\mathbb{R}^3} x \cdot \nabla S_0(\psi) |\psi|^2 \,dx
= - \frac{1}{2} \int_{\mathbb{R}^3} S_0(\psi) |\psi |^2 \,d x. 
\end{equation}
Similarly by the definition of $S_1$ and $S_2$ given in \ef{nonlocal2}
and \ef{nonlocal3} respectively, we arrive at
\begin{align} \label{eq:2.11}
& \int_{\mathbb{R}^3} x \cdot \nabla S_1 |\psi|^2 \,d x
= -\frac{1}{8 \pi} \sum_{j=1}^3 \int_{\mathbb{R}^3} x_j \frac{\partial}{\partial x_j}
\left( \int_{\mathbb{R}^3} \frac{\rho(y)}{|x-y|} \,dy \right) |\psi(x)|^2 \,d x \notag \\
&= \frac{1}{8 \pi} \sum_{j=1}^3 \int_{\R^3} \intR 
\frac{x_j (x_j-y_j)}{|x-y|^3} \rho(y) |\psi(x)|^2 \,d y \,d x \notag \\
&= \frac{1}{8 \pi} \int_{\mathbb{R}^3} \int_{\mathbb{R}^3} 
\frac{\rho(y) |\psi(x)|^2}{|x-y|} dy dx 
+\frac{1}{8 \pi} \sum_{j=1}^3 \int_{\R^3} \int_{\mathbb{R}^3} 
\frac{y_j(x_j- y_j)}{|x-y|^3} \rho(y) |\psi(x)|^2 dy dx \notag \\
&= -\int_{\mathbb{R}^3} S_1 |\psi |^2 \,dx
+ \frac{1}{8 \pi} \int_{\mathbb{R}^3} 
\left( \int_{\mathbb{R}^3} y \cdot \nabla_y \frac{1}{|x-y|} \rho(y) \,dy \right)
|\psi(x)|^2 \,d x \notag \\
&= -\int_{\mathbb{R}^3} S_1 |\psi|^2 \,dx 
+\frac{1}{8 \pi} \int_{\mathbb{R}^3} \left( \int_{\R^3} 
\Div_y \left( \frac{y \rho(y)}{|x-y|} \right) \,d y\right) |\psi(x)|^2 \,d x \notag \\
&\quad -\frac{1}{8 \pi} \int_{\mathbb{R}^3}
\left( \int_{\mathbb{R}^3} \Div_y y \frac{\rho(y)}{|x-y|} \,dy \right) |\psi(x)|^2 \,d x 
- \frac{1}{8 \pi} \int_{\R^3} \left( \intR 
\frac{y \cdot \nabla \rho(y)}{|x-y|} \,dy \right) |\psi(x)|^2 \,d x \notag \\
&= 2 \int_{\mathbb{R}^3} S_1|\psi|^2 \,dx 
-\int_{\mathbb{R}^3} S_2 |\psi|^2 \,d x.
\end{align}
From \ef{eq:2.5}-\ef{eq:2.11}, we obtain
\[
\begin{aligned}
0 &= \frac{1}{4} V^{\prime \prime}(t)
-2 \int_{\R^3} |\nabla \psi|^2 \,dx 
+ \frac{3(p-1)}{p+1} \int_{\R^3} |\psi|^{p+1} \,dx \\
&\quad -\frac{e^2}{2} \int_{\R^3} S_0(\psi)|\psi|^2 \,dx
+2 e^2 \intR S_1 |\psi|^2 \,dx 
-e^2 \int_{\R^3} S_2 |\psi|^2 \,d x \\
&= \frac{1}{4} V^{\prime \prime}(t)-2 Q(\psi),
\end{aligned}
\]
which completes the proof.

A rigorous proof can be carried out by using the density argument and the 
regularizing argument as in \cite[Proposition 6.5.1]{Ca}.
\end{proof}

\section{Variational setting and preliminaries}

The aim of this section is to prepare several properties of the energy functional
and present intermediate lemmas which will be used later on.

\subsection{Decomposition of the energy} \ 

In this subsection, we rewrite the energy functional $\mathcal{I}$
in a more convenient way.
We put 
\[
A(u)= \| \nabla u \|_2^2, \quad
B(u)= \| u \|_2^2, \quad
C(u)= \| u \|_{p+1}^{p+1},
\]
and decompose $\CI$ in the following way:
\[
\begin{aligned}
\CI(u)
&= \frac{1}{2} A(u) + \frac{\omega}{2} B(u) -\frac{1}{p+1} C(u) +\frac{e^2}{4} \intR S_0(u) |u|^2 \,dx \\
&\quad  
+\frac{e^2}{4} \intR S_1 |u|^2 \,dx 
-\frac{e^2}{4} \intR S_0(u) \rho(x) \,dx 
-\frac{e^2}{4} \intR S_1 \rho(x) \,dx. 
\end{aligned}
\]
Next we define three nonlocal terms:
\begin{align*}
D(u)&=\frac{1}{4} \intR S_0(u) |u|^2 \,dx, \\
E_1(u) &= -\frac{1}{4} \intR S_0(u) \rho(x) \,dx
=\frac{1}{4} \intR S_1 |u|^2 \,dx, \\
F &=-\frac{1}{4} \intR S_1 \rho(x) \,dx. 
\end{align*}
Note that $F$ is independent of $u \in \HT$.
Then we are able to write $\mathcal{I}$ defined in \ef{eq:1.4} in the following form:
\[
\mathcal{I}(u)= \frac{1}{2} A(u)
+ \frac{\omega}{2} B(u)
-\frac{1}{p+1} C(u) 
+e^2 D(u) + 2e^2 E_1(u) + e^2 F.
\]
Now it is convenient to define $I(u):=\mathcal{I}(u)-e^2 F$, which yields that
\begin{equation} \label{eq:3.1}
I(u) = \frac{1}{2}A(u) + \frac{\omega}{2} B(u) - \frac{1}{p+1}C(u) 
+ e^2 D(u) +2e^2 E_1(u).
\end{equation}
Since $F$ is independent of $u$, for minimization purpose,
we only have to work with the functional $I$. 
Recalling that 
\[
S_0(u)(x)=(-\Delta)^{-1} \left(\frac{|u(x)|^2}{2}\right) \geq 0, 
\]
we find that  
\begin{equation} \label{eq:3.2}
D(u) \ge 0 \quad \hbox{for all} \ u \in H^1(\R^3,\C).
\end{equation}
For later use, let us also define
\begin{align}
E_2(u) &:= \frac{1}{2} \intR S_0(u) x \cdot \nabla \rho (x) \,dx
= \frac{1}{2} \intR S_2 |u|^2 \,dx, \label{eq:3.3} \\
E_3(u) &:= \frac{1}{2} \intR S_0(u) x \cdot (D^2 \rho (x)x) \,dx
= \frac{1}{2} \intR S_3 |u|^2 \,dx, \notag \\
S_3(x) &= (-\Delta )^{-1} \left( \frac{x \cdot (D^2 \rho (x) x )}{2} \right)
= \frac{1}{8 \pi |x|} * \big( x \cdot (D^2 \rho(x) x) \big), \notag 
\end{align}
which is well-defined for $u \in H^1(\R^3,\C)$ by (A1).

\subsection{Estimates of nonlocal terms} \ 

This subsection is devoted to present
estimates for the nonlocal terms.

\begin{lemma} \label{lem:3.1}
For any $u \in H^1(\R^3,\C)$, 
$S_0$, $D$, $E_1$, $E_2$ and $E_3$ satisfy the estimates: 
\[
\begin{aligned}
\| S_0(u) \|_{6} 
&\le C\| \nabla S_0(u) \|_{2} 
\le C \| u\|_{{\frac{12}{5}}}^2
\le C \| u \|^2, \\
\| S_0(u) \|_6
&\le C \| u \|_2^{\frac{5p-7}{3(p-1)}} \| u \|_{p+1}^{\frac{p+1}{3(p-1)}}
\le C( \| u \|_2^2 + \| u \|_{p+1}^2)
\quad \text{if} \ \ 2<p<5, \\
D(u) &\leq C \| S_0(u) \|_6 \| u \|_{\frac{12}{5}}^2
\le C \| u \|^4, \\
|E_1(u)| &\le \frac{1}{4} \|S_0(u) \|_{6} \| \rho \|_{\frac{6}{5}} 
\le C \| \rho \|_{\frac{6}{5}} \| u \|^2, \\
|E_2(u)| &\le \frac{1}{2} \| S_0(u) \|_{6} 
\| x \cdot \nabla \rho \|_{\frac{6}{5}}
\le C \| x \cdot \nabla \rho \|_{\frac{6}{5}} \| u \|^2, \\
|E_3(u)| &\le \frac{1}{2} \| S_0(u) \|_{6} 
\| x \cdot (D^2 \rho x) \|_{\frac{6}{5}}
\le C \| x \cdot (D^2 \rho x) \|_{\frac{6}{5}} \| u \|^2.
\end{aligned}
\]

\end{lemma}
For the proof of the inequality on $S_0(u)$, we refer to \cite{R}.
The other estimates can be obtained by the H\"odler inequality
and the Sobolev inequality.

\subsection{Scaling properties} \

In this subsection, we collect scaling properties of the nonlocal terms. 
For $a$, $b \in \R$ and $\la >0$, 
let us adapt the scaling $u_{\la} (x) := \la^a u( \la^b x)$. 
Then we have
\[
\begin{aligned}
S_0(u_{\la})(x)
&=\frac{1}{8 \pi } \intR \frac{|u_{\la}(y)|^2}{\left| x-y\right|} \,d y 
=\frac{\la^{2a}}{8 \pi } 
\intR \frac{|u(\la^b y)|^2}{\left| x- y\right|} \,dy  
\stackrel{y=\la^{-b} z}{=} \frac{\la^{2a-2b}}{8\pi} \intR 
\frac{|u(z)|^2}{|\la^b x -z|} \,dz. 
\end{aligned}
\]
Thus one finds that
\begin{align}
S_0(u_{\la})(x) &= \la^{2 a -2 b} S_0(u) ( \la^{b} x), \nonumber \\
D(u_\la) &= \la^{4 a-5 b} D(u),  \label{eq:3.4} \\
E_1(u_\la) 
& =-\frac{1}{4} \intR S_0(u_{\la}) (x) \rho(x) \,d x 
=- \frac{\la^{2 a-2 b}}{4} \intR S_0(u)(\la^{b} x) \rho( x) \,dx \nonumber \\
&=- \frac{\la^{2 a-5 b}}{4} \intR S_0(u)(x) \rho (\la^{-b} x) \,d x. \label{eq:3.5}
\end{align}
By the H\"older inequality, it follows that
\begin{equation} \label{eq:3.6}
|E_1(u_{\la})| \le \frac{\la^{2a-5b}}{4} 
\| S_0(u)\|_6 \| \rho( \la^{-b} x ) \|_{\frac{6}{5}}
\le C \la^{2a-\frac{5}{2}b} \| \rho \|_{\frac{6}{5}} \| S_0(u) \|_6.
\end{equation}
Similarly, we have
\begin{align} 
E_2(u_{\la}) &= \frac{\la^{2a-5b}}{2} \intR S_0(u)(x)
(\la^{-b}x) \cdot \nabla \rho (\la^{-b}x) \,dx, \label{eq:3.7} \\
|E_2(u_{\la})| &\le \frac{\la^{2a-5b}}{2} 
\| S_0(u)\|_6 \| (\la^{-b} x) \cdot \nabla \rho( \la^{-b} x ) \|_{\frac{6}{5}} 
\le C \la^{2a-\frac{5}{2}b} \| x \cdot \nabla \rho \|_{\frac{6}{5}} \| S_0(u) \|_6. 
\label{eq:3.8}
\end{align}

\subsection{Nehari and Pohozaev identities} \ 

This subsection is devoted to introduce the
Nehari identity and the Pohozaev identity associated with \ef{eq:1.1}.

\begin{lemma} \label{lem:3.2}
Let $u \in H^1(\R^3,\C)$ be a weak solution of \ef{eq:1.1}.
Then $u$ satisfies the Nehari identity $N(u)=0$ and 
the Pohozaev identity $P(u)=0$, where
\begin{align}
N(u)&= A(u) + \omega B(u) - C(u) + 4e^2 D(u) + 4e^2 E_1(u), \label{eq:3.9} \\
P(u)&= \frac{1}{2} A(u) + \frac{3\omega}{2} B(u) - \frac{3}{p+1} C(u)
+5e^2 D(u) + 10e^2 E_1(u) - e^2 E_2(u). \label{eq:3.10}
\end{align}

\end{lemma}

For the proof, we refer to \cite{CW4, CW5}.
Since $Q(u)= \frac{3}{2}N(u)-P(u)$, 
the functional $Q$ defined in \ef{energy} can be also written as follows:
\begin{equation} \label{eq:3.11}
Q(u) = A(u)- \frac{3(p-1)}{2(p+1)} C(u)
+e^2 D(u) - 4e^2 E_1(u) + e^2 E_2(u).
\end{equation}

\section{Properties of ground state solutions}

In this section, we introduce fundamental properties 
of ground state solutions of \ef{eq:1.1}.
Now let us define
\[
J(u):= 2N(u) - P(u).
\]
From \ef{eq:3.9} and \ef{eq:3.10}, for any $u\in H^1(\R^3,\C)$, it holds that
\begin{equation} \label{eq:4.1}
J(u) =\frac{3}{2} A(u)+\frac{\omega}{2} B(u)
-\frac{2p-1}{p+1} C(u) +3 e^2 D(u)-2 e^2 E_1(u)+e^2 E_2(u).
\end{equation}
We also denote by $\CM$ the Nehari-Pohozaev set:
\[
\CM =\left\{ u \in \HT \setminus \{ 0 \} \mid J(u)=0 \right\} .
\]
By Lemma \ref{lem:3.2}, one knows that 
any nontrivial weak solution of \ef{eq:1.1} belongs to $\CM$.
Next let us define 
\begin{equation} \label{eq:4.2}
\sigma := \inf_{u \in \CM} I(u).
\end{equation}
and the ground state energy level for \ef{eq:1.1} by
\[
m := \inf_{u \in \CS} I(u),
\quad \CS =\left\{ u \in \HT \setminus \{0 \} 
\mid I'(u)=0 \right\}.
\]
Note that $u$ is a weak solution of $\ef{eq:1.1}$ if and only if $ I'(u)=0$.
Then we have the following proposition. 

\begin{proposition} \label{prop:4.1}
Suppose that $2<p<5$. 
Assume {\rm(A1)}, {\rm(A2)} and 
\[
e^2 \left( \| \rho \|_{\frac{6}{5}} +\|x \cdot \nabla \rho \|_{\frac{6}{5}}
+\| x \cdot (D^2 \rho x) \|_{\frac{6}{5}}\right) \le \rho_0
\]
for sufficiently small $\rho_0>0$.

\begin{enumerate}
\item[\rm (i)] 
If $\sigma$ defined in \ef{eq:4.2} is attained by some $u_0 \in \HT \setminus \{0\}$,
then $u_0$ is a ground state solution of \ef{eq:1.1}, namely $u_0$ satisfies
\[
\sigma = I(u_0)= m.
\]

\item[\rm (ii)] There exists $u_0 \in \HT \setminus \{0 \}$ such that
\[
I(u_0)= \sigma \quad \text{and} \quad J(u_0)=0.
\]

\item[\rm (iii)]Any ground state solution of \ef{eq:1.1} is real-valued up to phase shift.

\end{enumerate}

\end{proposition}

For the proof of Proposition \ref{prop:4.1}, we refer to \cite{CW5}.

\smallskip
Next we establish an exponential decay of ground state solution $u_0$ of \ef{eq:1.1},
which guarantee that $u_0 \in \Sigma$.
For this purpose, we apply the following result in \cite{BV}.

\begin{proposition} \label{prop:4.2}
Let $V:\R^N \to \R_+$, $f \in L^{\frac{r}{r-2}}(\R^N)$ for $r>2$
and $u \in H^1_V(\RN,\R)$, where
$$
H^1_V(\R^N)=\left\{
u \in W^{1,1}_{loc}(\RN) \ ; \ 
\| u\|_{H^1_V}^2 := 
\int_{\R^N} \big( | \nabla u|^2+V(x) |u|^2 \big) \,dx < \infty \right\}.$$
Assume that the embedding $H^1_V \hookrightarrow L^r(\RN)$ is continuous
and $u$ satisfies
\begin{equation} \label{eq:4.3}
-\Delta u+V(x) u= f(x) u \quad \hbox{in} \ \R^N.
\end{equation}
If there exist $\beta>0$ and $k>0$ such that
$$
\liminf_{|x| \to \infty} V(x)|x|^{2-2\beta} > k^2,$$
then there exists $C>0$ depending on $\beta$, $k$, $\| f\|_{L^{\frac{r}{r-2}}}$ and
$\| u\|_{H^1_V}$ such that
$$
|u(x)| \le Ce^{-k (1+|x|)^{\beta}} \quad \hbox{for all} \ x\in \R^N.$$
\end{proposition}

An important consequence of Proposition \ref{prop:4.2} is given 
in the next lemma and concerns the exponential decay estimate of $u_0$.

\begin{lemma} \label{lem:4.3}
Suppose that $2<p<5$.
Assume {\rm(A1)}-{\rm(A3)} and 
\[
e^2 \left( \| \rho \|_{\frac{6}{5}} +\|x \cdot \nabla \rho \|_{\frac{6}{5}}
+\| x \cdot (D^2 \rho x) \|_{\frac{6}{5}}\right) \le \rho_0
\]
for sufficiently small $\rho_0>0$.
Let $u_0$ be the ground state solution of \ef{eq:1.1}
given in Proposition \ref{prop:4.1}.
Then for any $k < \sqrt{\omega}$,
there exists $C>0$ such that 
\[
u_0(x) \le Ce^{-k(1+|x|)} \quad \hbox{for all} \ x\in \R^3.
\]
\end{lemma}

\begin{proof}
Let us take $\beta=1$, $r=p+1$, 
$H_V^1=H^1(\R^3)$, $f(x)=|u_0(x)|^{p-1}$ and
\[
V(x)= \omega +e^2 S(u_0)(x) = \omega + e^2 S_0(u_0) + e^2 S_1(x),
\]
so that $u_0$ satisfies \ef{eq:4.3} and $H_V^1 \hookrightarrow L^r(\R^3)$.
To apply Proposition \ref{prop:4.2}, we show that
\begin{equation} \label{eq:4.4}
\liminf_{|x| \to \infty} V(x) \ge \omega.
\end{equation}

First we observe from \ef{nonlocal1} that $S(u_0)(x) \ge 0$.
Next we claim that
\begin{equation} \label{eq:4.5}
\lim_{|x| \to \infty} S_1(x) =0,
\end{equation}
which yields that \ef{eq:4.4} holds.
Indeed for a fixed $x \in \R^3$, it follows from \ef{nonlocal2} that
\[
|S_1(x)| \le \frac{1}{8 \pi} \intR \frac{| \rho(y)|}{|x-y|} \,dy
= \frac{1}{8 \pi} \int_{|y| \le \frac{1}{2}|x|} \frac{| \rho(y)|}{|x-y|} \,dy
+\frac{1}{8 \pi} \int_{|y| \ge \frac{1}{2}|x|} \frac{| \rho(y)|}{|x-y|} \,dy.
\]
By the assumption (A3) and the fact $\alpha>2$,
one can take $q$ so that $\rho \in L^q(\R^3)$ and
$\max \{ 1, \frac{3}{\alpha} \} < q < \frac{3}{2}$.
This also implies that the H\"older conjugate $q'$ satisfies $q'>3$.
Then observing that
\[
|y| \le \frac{1}{2} |x| \quad \Rightarrow \quad 
|x-y| \ge |x| - |y| \ge \frac{1}{2}|x|
\]
and by the H\"older inequality, we have
\begin{align*}
\int_{|y| \le \frac{1}{2}|x|} \frac{| \rho(y)|}{|x-y|} \,dy
&\le \frac{2}{|x|} \int_{|y| \le \frac{1}{2} |x|} | \rho(y)| \,dy \\
&\le \frac{2}{|x|} 
\left( \int_{|y| \le \frac{1}{2} |x|} \,dy \right)^{\frac{1}{q'}}
\left( \int_{|y| \le \frac{1}{2} |x|} |\rho(y)|^q \,dy \right)^{\frac{1}{q}} \\
&\le \frac{C}{|x|^{1-\frac{3}{q'}}} \| \rho \|_{L^q(\R^3)} \to 0 
\quad \text{as} \ |x| \to \infty.
\end{align*}
Next we decompose
\[
\begin{aligned}
&\int_{|y| \ge \frac{1}{2} |x|} \frac{|\rho(y)|}{|x-y|} \,dy 
= \int_{|y| \ge \frac{1}{2} |x|, |x-y| \le \frac{1}{2}|x|} \frac{|\rho(y)|}{|x-y|} \,dy
+ \int_{|y| \ge \frac{1}{2} |x|, |x-y| \ge \frac{1}{2}|x|} \frac{|\rho(y)|}{|x-y|} \,dy.
\end{aligned}
\]
Then from (A3), one finds that
\begin{align*}
\int_{|y| \ge \frac{1}{2} |x|, |x-y| \le \frac{1}{2}|x|} \frac{|\rho(y)|}{|x-y|} \,dy
& \le \int_{|y| \ge \frac{1}{2}|x|, |x-y| \le \frac{1}{2}|x|}
\frac{1}{|x-y|} \cdot \frac{C}{1+|y|^{\alpha}} \,dy \\
&\le \frac{C}{|x|^{\alpha}} 
\int_{|x-y| \le \frac{1}{2}|x|} \frac{1}{|x-y|} \,dy
\le \frac{C}{|x|^{\alpha-2}} \to 0 \quad \text{as} \ |x| \to \infty.
\end{align*}
Moreover by the H\"older inequality, we also have
\begin{align*}
\int_{|y| \ge \frac{1}{2} |x|, |x-y| \ge \frac{1}{2}|x|} \frac{|\rho(y)|}{|x-y|} \,dy 
&\le \left( \int_{|y| \ge \frac{1}{2}|x|} |\rho(y)|^q \,dy \right)^{\frac{1}{q}}
\left( \int_{|x-y| \ge \frac{1}{2}|x|} \frac{1}{|x-y|^{q'}} \,dy \right)^{\frac{1}{q'}} \\
&\le \frac{C}{|x|^{1-\frac{3}{q'}}} \| \rho \|_{L^q(\R^3)} 
\to 0 \quad \text{as} \ |x| \to \infty.
\end{align*}
Thus we obtain \ef{eq:4.5}.

Now from \ef{eq:4.4}, we are able to apply Proposition \ref{prop:4.2} 
to obtain the desired decay estimate. 
\end{proof}

\section{Fibering maps along $L^2$-invariant scaling}

In this section, we consider the $L^2$-invariant scaling:
\begin{equation} \label{eq:5.1}
u^{\la} (x) = \la^{\frac{3}{2}} u(\la x)
\end{equation}
and investigate several fibering maps along this curve.
First we begin with the following lemma.

\begin{lemma} \label{lem:5.1}
Suppose that $2<p<5$ and $J(u) \le 0$.
There exist $\rho_0$ and $C_0>0$ independent of $e$, $\rho$ such that
if $e^2 \left( \| \rho \|_{\frac{6}{5}} + \| x \cdot \nabla \rho \|_{\frac{6}{5}} \right) 
\le \rho_0$,
then it holds that
\[
\| u \|^2 \le C_0 \| u \|_{p+1}^{p+1}.
\]

\end{lemma}

\begin{proof}
From \ef{eq:3.2} and \ef{eq:4.1} and by Lemma \ref{lem:3.1}, one has
\[
\begin{aligned}
0 &\ge J(u) 
\ge \frac{\min\{ 3, \omega \}}{2} \| u \|^2
- \frac{2p-1}{p+1} \| u \|_{p+1}^{p+1} 
- C_1 e^2 \left( \| \rho \|_{\frac{6}{5}} + \| x \cdot \nabla \rho \|_{\frac{6}{5}} \right)
\| S_0(u) \|_6
\end{aligned}
\]
for some $C_1>0$. 
By Lemma \ref{lem:3.1}, we also have
$\| S_0(u) \|_6 \le C_2 \| u \|^2$, which shows that
\[
0 \ge \left\{ \frac{\min \{ 3, \omega \}}{2} - C_1 C_2
e^2 \left( \| \rho \|_{\frac{6}{5}} + \| x \cdot \nabla \rho \|_{\frac{6}{5}} \right) \right\}
\| u \|^2 - \frac{2p-1}{p+1} \| u \|_{p+1}^{p+1}.
\]
Choosing $\rho_0 = \frac{\min\{ 3, \omega \}}{4 C_1 C_2}$ and taking
$e^2 \left( \| \rho \|_{\frac{6}{5}} + \| x \cdot \nabla \rho \|_{\frac{6}{5}} \right) 
\le \rho_0$,
we deduce that
\[
0 \ge \frac{\min \{ 3, \omega \}}{4} \| u \|^2
- \frac{2p-1}{p+1} \| u \|_{p+1}^{p+1},
\]
which ends the proof.
\end{proof}

Next we consider the fibering map $J(u^{\la})$ for \ef{eq:5.1}.
Applying \ef{eq:3.4} with $a=\frac{3}{2}$, $b=1$
and using \ef{eq:4.1}, we find that
\begin{equation} \label{eq:5.2}
J(u^{\la})
= \frac{3\la^2}{2} A(u) + \frac{\omega}{2} B(u) 
-\frac{2p-1}{p+1} \la^{\frac{3(p-1)}{2}} C(u) + 3e^2 \la D(u) 
- 2e^2 E_1(u^{\la}) + e^2 E_2(u^{\la}).
\end{equation}

\begin{lemma} \label{lem:5.2}
Suppose $\frac{7}{3} < p <5$ and let $u \in H^1(\R^3, \C) \setminus \{ 0 \}$ 
satisfy $J(u) \le 0$.
Then there exists $\la^* = \la^*_{e,\rho} \in (0,1]$ such that $J(u^{\la^*})=0$.
Moreover there exist $\rho_0>0$ and $\delta^* \in (0,1)$
independent of $e$, $\rho$ such that 
if $e^2 \left( \| \rho \|_{\frac{6}{5}} + \| x \cdot \nabla \rho \|_{\frac{6}{5}} \right) 
\le \rho_0$, it holds that $\la^* \ge \delta^*$.
\end{lemma}

\begin{proof}
First by using \ef{eq:3.6} and \ef{eq:3.8}, one has
\[
e^2 |E_1(u^{\la})| + e^2 |E_2(u^{\la})|
\le C \la^{\frac{1}{2}} 
e^2 \big( \| \rho \|_{\frac{6}{5}} + \| x \cdot \nabla \rho \|_{\frac{6}{5}} \big) 
\| S_0(u) \|_6.
\]
Thus it follows that
\[
\lim_{\la \to 0+} J(u^{\la}) = \frac{\omega}{2} B(u) >0.
\]
Moreover, for $\lambda=1$, one has $J(u^{\la})=J(u)\leq 0$.
By the continuity of $J(u^{\la})$ with respect to $\la$, 
there exists $\tilde{\la} \in (0,1]$ such that $J(u^{\tilde{\la}})=0$.
We then define
\[
\la^* := \inf \left\{ \la\in(0,1] \mid J(u^\la)=0\right\}.
\]
It is obvious that $\la^{*}\in(0,1]$ and $J(u^{\la^*})=0$.

Next we establish the uniform lower estimate of $\la^*$.
For this purpose, let us define the Nehari-Pohozaev functional for NLS
(that is, we take $e=0$ in \ef{eq:1.1}) by
\[
J_0(u) := \frac{3}{2} A(u) + \frac{\omega}{2} B(u)
- \frac{2 p-1}{p+1} C(u), \quad u \in H^1(\R^3, \C).
\]
We claim that $J_0(u) \le 0$ if 
$e^2 \left( \| \rho \|_{\frac{6}{5}} + \| x \cdot \nabla \rho \|_{\frac{6}{5}} \right) \ll 1$.
In fact by Lemma \ref{lem:3.1} and \ef{eq:3.2}, it follows that
\[
\begin{aligned}
0 \ge J(u) &= J_0(u) + 3 e^2 D(u) - 2 e^2 E_1(u) + e^2 E_2(u) \\
&\ge J_0(u) - C
e^2 \left( \| \rho \|_{\frac{6}{5}} + \| x \cdot \nabla \rho \|_{\frac{6}{5}} \right) 
\| S_0(u) \|_6.
\end{aligned}
\]
Thus if
\[
e^2 \left( \| \rho \|_{\frac{6}{5}} + \| x \cdot \nabla \rho \|_{\frac{6}{5}} \right)
\le \frac{ |J_0(u)|}{2 C \| S_0(u) \|_6},
\]
we deduce that
\[
0 \ge J_0(u) - \frac{1}{2} | J_0(u)| \quad \text{and hence} \quad J_0(u) \le 0.
\]
Since $J_0(u) \le 0$ and 
\[
J_0(u^{\la}) = \frac{3\la^2}{2} A(u) + \frac{\omega}{2} B(u)
- \frac{2p-1}{p+1} \la^{\frac{3(p-1)}{2}} C(u),
\]
we can readily see that there exists a unique $\la_0 \in (0,1]$ such that
$J_0(u^{\la_0})=0$.

Finally we prove that $\la^* \ge \frac{\la_0}{2}$ if 
$e^2 \left( \| \rho \|_{\frac{6}{5}} + \| x \cdot \nabla \rho \|_{\frac{6}{5}} \right) \ll 1$.
Indeed since $J_0(u^{\frac{\la_0}{2}})>0$, we find that
\[
J(u^{\frac{\la_0}{2}}) \ge J_0( u^{\frac{\la_0}{2}})
-C \left( \frac{\la_0}{2} \right)^{\frac{1}{2}}
e^2 \left( \| \rho \|_{\frac{6}{5}} + \| x \cdot \nabla \rho \|_{\frac{6}{5}} \right) 
\| S_0(u) \|_6
>0
\]
provided that 
$e^2 \left( \| \rho \|_{\frac{6}{5}} + \| x \cdot \nabla \rho \|_{\frac{6}{5}} \right)$
is sufficiently small.
This implies that $\la^* \ge \frac{\la_0}{2}$, which completes the proof.
\end{proof}

Next we investigate another fibering map:
\begin{equation} \label{eq:5.3}
f(\la) := I(u^{\la}) - \frac{\la^2}{2} Q(u).
\end{equation}
From \ef{eq:3.1} and \ef{eq:3.11}, it follows that
\begin{align} \label{eq:5.4}
f(\la) &= \frac{\omega}{2} B(u) 
- \frac{1}{4(p+1)} \left( 4 \la^{\frac{3(p-1)}{2}} - 3(p-1) \la^2 \right) C(u)
- \frac{e^2}{2} (\la^2-2\la) D(u) \notag \\
&\quad + 2e^2 E_1(u^{\la}) + 2e^2 \la^2 E_1(u) - \frac{e^2\la^2}{2} E_2(u).
\end{align}
Then we have the following energy inequality,
which is a key in our analysis.

\begin{lemma} \label{lem:5.3}
Suppose that $\frac{7}{3} < p < 5$.
There exist $C_1$, $C_2>0$ independent of $e$, $\rho$, $\la$ such that
the following estimate holds:
For any $u \in H^1(\R^3, \C)$ and $\la \in [\delta^*,1]$, 
where $\delta^* \in (0,1)$ is the constant in Lemma \ref{lem:5.2},
it holds that
\begin{equation} \label{eq:5.5}
\begin{split}
f(\la) - f(1) 
&\le - C_1 (1-\la)^2 \| u \|_{p+1}^{p+1} \\
&\quad + C_2(1-\la)^2 e^2 \left( \| \rho \|_{\frac{6}{5}} + \| x \cdot \nabla \rho \|_{\frac{6}{5}} 
+ \| x \cdot (D^2 \rho x) \|_{\frac{6}{5}} \right) \| u \|^2. 
\end{split}
\end{equation}

\end{lemma}

\begin{proof}
The proof consists of four steps.

\noindent
\textbf{Step 1} (Transformation of $f(\la) -f(1)$): 
First we observe from \ef{eq:3.2} and \ef{eq:5.4} that
\[
\begin{aligned}
f(\la)-f(1) 
&= - \frac{1}{4(p+1)} \left( 4 \la^{\frac{3(p-1)}{2}} - 3(p-1) \la^2 + 3p-7 \right) C(u)
- \frac{e^2}{2} (1- \la)^2 D(u) \\
&\quad +2e^2 E_1(u^{\la}) 
+2 e^2(\la^2-2) E_1(u) -\frac{e^2}{2} (\la^2-1) E_2(u) \\
&\le - \frac{1}{4(p+1)} \left( 4 \la^{\frac{3(p-1)}{2}} - 3(p-1) \la^2 + 3p-7 \right) C(u) \\
&\quad +2e^2 E_1(u^{\la}) 
+2 e^2(\la^2-2) E_1(u) -\frac{e^2}{2} (\la^2-1) E_2(u).
\end{aligned}
\]
Moreover recalling \ef{eq:3.5}, \ef{eq:3.7} and putting
\[
\begin{aligned}
R(\la,u) &:= 2e^2 E_1(u^{\la}) 
+2 e^2(\la^2-2) E_1(u) -\frac{e^2}{2} (\la^2-1) E_2(u) 
= - \frac{e^2}{4} \intR S_0(u) M(\la,x) \,dx, \\
M(\la, x) &:= 2 \la^{-2} \rho( \la^{-1} x) 
+2 (\la^2 -2) \rho(x) + (\la^2 -1) x \cdot \nabla \rho(x),
\end{aligned}
\]
we arrive at
\begin{equation} \label{eq:5.6}
f(\la)-f(1) 
\le - \frac{1}{4(p+1)} \left( 4 \la^{\frac{3(p-1)}{2}} - 3(p-1) \la^2 + 3p-7 \right) 
\| u \|_{p+1}^{p+1} + R(\la,u).
\end{equation}

\noindent
\textbf{Step 2} (Evaluation of coefficients): Let us put
\[
g(\la) := 4 \la^{\frac{3(p-1)}{2}} - 3(p-1) \la^2 + 3p-7 
\quad \text{for} \ \la \in [\delta^*,1].
\]
Then one finds that $g(1)=g'(1)=0$.
Thus by the Taylor theorem, we have
\[
g(\la) = \frac{1}{2} g''(\xi)(1-\la)^2 \quad
\text{for some $\xi$ between $\la$ and $1$}.
\]
Now we choose $\tau \in (0,1)$ such that
$(1-\tau)^{\frac{3 p-7}{2}}=\frac{3(p-1)}{2(3p-5)}$.
Then for $\la \in [1-\tau,1]$, it follows that
\begin{align*}
g''(\la) &= 3(p-1)(3p-5) \la^{\frac{3p-7}{2}} - 6(p-1) \\
&\ge 3(p-1)(3p-5) (1-\tau)^{\frac{3p-7}{2}} - 6(p-1) 
= \frac{3}{2} (p-1)(3p-7) >0,
\end{align*}
yielding that
\[
g(\la) \ge \frac{3}{4} (p-1)(3p-7) (1-\la)^2 
\quad \text{for} \ 1-\tau \le \la \le 1.
\]
Moreover since $g'(\la)=6(p-1) \la ( \la^{\frac{3p-7}{2}} -1) <0$ on $[0,1)$,
we have \[
g(\la) \ge g(1-\tau) \ge \frac{g(1-\tau)}{(1-\delta^*)^2} (1-\la)^2
\quad \text{for} \ \delta^* \le \la \le 1-\tau.
\]
Thus putting
\[
C_1 := \frac{1}{4(p+1)} \min \left\{ 
\frac{3}{4}(p-1)(3p-7), \frac{g(1-\tau)}{(1-\delta^*)^2} \right\},
\]
we conclude that 
\begin{equation} \label{eq:5.7}
g(\la) \ge 4(p+1) C_1 (1-\la)^2 \quad \text{on} \ [\delta^*,1].
\end{equation}

\noindent
\textbf{Step 3} (Estimate of $R(\la,u)$): First we find that
\begin{align*}
\frac{\partial M}{\partial \la} 
&= -4\la^{-3} \rho(\la^{-1}x) -2 \la^{-4} x \cdot \nabla \rho( \la^{-1} x)
+ 4 \la \rho(x) + 2 \la x \cdot \nabla \rho(x), \\
\frac{\partial^2 M}{\partial \la^2}
&= 12 \la^{-4} \rho(\la^{-1} x) +12 \la^{-5} x \cdot \nabla \rho ( \la^{-1} x)
+ 2 \la^{-6} x \cdot \big( D^2 \rho( \la^{-1} x) x \big) 
+ 4 \rho(x) + 2 x \cdot \nabla \rho(x). 
\end{align*}
Then for fixed $x \in \R^3$, it follows that 
$M(1,x) = \frac{\partial M}{\partial \la}(1,x)=0$ and
\begin{align*}
\frac{\partial^2 M}{\partial \la^2}(\la,x)
&\ge - 12 \la^{-4} \left| \rho(\la^{-1} x) \right| 
-12 \la^{-4} \left| (\la^{-1}x) \cdot \nabla \rho( \la^{-1} x) \right| \\
&\quad -2 \la^{-4} \left| (\la^{-1} x) \cdot 
\big( D^2 \rho( \la^{-1} x) (\la^{-1} x) \big) \right| 
-4 | \rho(x)| -2 | x \cdot \nabla \rho(x)| \\
&=: - N(\la, x).
\end{align*}
By the Taylor theorem, there exists $\xi = \xi (\la,x) \in ( \delta^*, 1)$ such that
\[
M(\la,x) = \frac{1}{2} \frac{\partial^2 M}{\partial \la^2}(\xi ,x) (1-\la)^2
\ge - \frac{1}{2} N(\xi, x) (1-\la)^2.
\]
Using the H\"older inequality, we deduce that
\[
R(\la, u) = -\frac{e^2}{4} \intR S_0(u) M(\la, x) \,d x 
\le \frac{e^2}{4} (1-\la)^2 \| S_0(u) \|_6 
\| N(\xi,x) \|_{\frac{6}{5}} 
\]
for $\delta^* \le \la \le 1$.
Moreover since $\delta^* < \xi <1$, one obtains
\[
\begin{aligned}
\| N(\xi , x) \|_{\frac{6}{5}} 
&\le 12 \left\| \xi^{-\frac{3}{2}} \rho \right\|_{\frac{6}{5}}
+ 12 \left\| \xi^{-\frac{3}{2}} x \cdot \nabla \rho \right\|_{\frac{6}{5}} 
+2 \left\| \xi^{-\frac{3}{2}} x \cdot (D^2 \rho x) \right\|_{\frac{6}{5}} 
+4 \| \rho \|_{\frac{6}{5}} + 2 \| x \cdot \nabla \rho \|_{\frac{6}{5}} \\
&\le \left( \frac{12}{(\delta^*)^{\frac{3}{2}}} + 4 \right) \| \rho \|_{\frac{6}{5}}
+ \left( \frac{12}{(\delta^*)^{\frac{3}{2}}} + 2 \right) 
\| x \cdot \nabla \rho \|_{\frac{6}{5}}
+ \frac{2}{(\delta^*)^{\frac{3}{2}}} \| x \cdot (D^2 \rho x) \|_{\frac{6}{5}},
\end{aligned}
\]
from which we conclude that
for $\delta^* \le \la \le 1$,
\begin{equation} \label{eq:5.8}
R(\la, u) \le C_2 (1-\la)^2 e^2
\left( \| \rho \|_{\frac{6}{5}} + \| x \cdot \nabla \rho \|_{\frac{6}{5}}
+ \| x \cdot (D^2 \rho x) \|_{\frac{6}{5}} \right) \| u \|^2
\end{equation}
by Lemma \ref{lem:3.1}, 
where $C_2>0$ is a constant independent of $e$, $\rho$ and $\la$.

\smallskip
\noindent
\textbf{Step 4} (Conclusion): 
Now from \ef{eq:5.6}-\ef{eq:5.8}, we can see that \ef{eq:5.5} holds.
\end{proof}

By Lemma \ref{lem:5.3}, we are able to prove the following proposition,
which plays an important role for the strong instability.

\begin{proposition} \label{prop:5.4}
Suppose that $\frac{7}{3} < p <5$ and let $u \in H^1(\R^3,\C) \setminus \{ 0 \}$
satisfy $J(u) \le 0$ and $Q(u) \le 0$.
Then there exists $\rho_0$ such that if
$e^2 \left( \| \rho \|_{\frac{6}{5}} + \| x \cdot \nabla \rho \|_{\frac{6}{5}}
+ \| x \cdot (D^2 \rho x) \|_{\frac{6}{5}} \right) \le \rho_0$,
it holds that
\begin{equation} \label{eq:5.9}
\frac{1}{2} Q(u) \le I(u) - I(u_0),
\end{equation}
where $u_0$ is the ground state solution of \ef{eq:1.1}.
\end{proposition}

\begin{proof}
First we claim that $f(\la^{*}) \le f(1)$, 
where $\la^* \in (0,1]$ is the constant given in Lemma \ref{lem:5.2}.
In fact by Lemma \ref{lem:5.1} and Lemma \ref{lem:5.3}, we find that
\begin{align*}
f(\la^*)-f(1) &\le - C_1(1-\la^*)^2 \| u \|_{p+1}^{p+1} \\
&\quad + C_2 (1-\la^*)^2 e^2 
\left( \| \rho \|_{\frac{6}{5}} + \| x \cdot \nabla \rho \|_{\frac{6}{5}}
+ \| x \cdot (D^2 \rho x) \|_{\frac{6}{5}} \right) \| u \|^2 \\
&\le - (1- \la^*)^2 \| u \|_{p+1}^{p+1} 
\left\{ C_1- C_0 C_2 e^2 
\left( \| \rho \|_{\frac{6}{5}} + \| x \cdot \nabla \rho \|_{\frac{6}{5}}
+ \| x \cdot (D^2 \rho x) \|_{\frac{6}{5}} \right) \right\}.
\end{align*}
Thus taking $\rho_0$ smaller, we obtain $f(\la^*) \le f(1)$ as claimed.

Now since $J(u^{\la^*})=0$, Proposition \ref{prop:4.1} shows that
\[
I(u_0) = \inf \{ I(u) \mid u \in H^1(\R^3, \C)\setminus \{ 0 \}, J(u) =0 \}
\le I(u^{\la^*}).
\]
Thus from \ef{eq:5.3} and $Q(u) \le 0$, we deduce that
\[
I(u_0) \le I(u^{\la^*}) = f(\la^*) + \frac{(\la^*)^2}{2} Q(u) 
\le f(\la^*) \le f(1) = I(u) - \frac{1}{2} Q(u).
\]
This completes the proof.
\end{proof}

\begin{remark} \label{rem:5.5}
By Lemma \ref{lem:3.1}, we know that $\| S_0(u) \|_6$ is controlled by
$\| u \|_2^2 + \| u \|_{p+1}^2$.
Since there is no term involving $\| u \|_2^2$ in $f(\la)-f(1)$
and the power of $\| u \|_{p+1}$ is different between $C(u)$ and $S_0(u)$,
we cannot expect that \ef{eq:5.9} holds without the assumption $J(u) \le 0$. 
\end{remark}

Next we consider two fibering maps:
\begin{align} \label{eq:5.10}
F(\la) &:= I(u^{\la}) \\
&= \frac{\la^2}{2} A(u) + \frac{\omega}{2} B(u) - \frac{\la^{\frac{3(p-1)}{2}}}{p+1} C(u)
+e^2 \la D(u)  
- \frac{e^2 \la^{-2}}{2} \intR S_0(u) \rho(\la^{-1} x) dx, \notag \\
G(\la) &:= Q(u^{\la}) \notag \\
&= \la^2 A(u) - \frac{3(p-1)}{2(p+1)} \la^{\frac{3(p-1)}{2}} C(u) + e^2 \la D(u) \notag \\
&\quad +e^2 \la^{-2} \intR S_0(u) \rho(\la^{-1} x) \,dx
+ \frac{e^2\la^{-2}}{2} \intR S_0(u) 
(\la^{-1} x) \cdot \nabla \rho (\la^{-1} x) \,dx. \notag
\end{align}
Then one finds that
\begin{equation} \label{eq:5.11}
G(\la) = \la F'(\la).
\end{equation}

\begin{lemma} \label{lem:5.6}
Suppose that $\frac{7}{3} < p < 5$ and let $u \in H^1(\R^3, \C) \setminus \{ 0 \}$
satisfy $J(u) \le 0$ and $Q(u) \le 0$.
There exists $\rho_0>0$ such that if 
$e^2 \left( \| \rho \|_{\frac{6}{5}} + \| x \cdot \nabla \rho \|_{\frac{6}{5}}
+ \| x \cdot (D^2 \rho x) \|_{\frac{6}{5}} \right) \le \rho_0$,
it holds that $F''(\la) <0$ for all $\la \in [1,\infty)$.
\end{lemma}

\begin{proof}
First from \ef{eq:5.10} and by the H\"older inequality, we find that
\begin{align*}
F''(\la) 
&= A(u) - \frac{3(p-1)(3p-5)}{4(p+1)} \la^{\frac{3p-7}{2}} C(u)
-3e^2 \la^{-4} \intR S_0(u) \rho( \la^{-1} x) \,dx \\
&\quad - 3e^2 \la^{-4} \intR S_0(u) (\la^{-1} x) \cdot \nabla \rho( \la^{-1} x) \,dx \\
&\quad - \frac{e^2 \la^{-4}}{2} \intR S_0(u) 
\left( \la^{-1}x) \cdot (D^2 \rho(\la^{-1} x) (\la^{-1} x) \right) dx \\
&\le A(u) - \frac{3(p-1)(3p-5)}{4(p+1)} \la^{\frac{3p-7}{2}} C(u) \\
&\quad + C e^2 \la^{-\frac{3}{2}}
\left( \| \rho \|_{\frac{6}{5}} + \| x \cdot \nabla \rho \|_{\frac{6}{5}}
+ \| x \cdot (D^2 \rho x) \|_{\frac{6}{5}} \right) \| S_0(u) \|_6 \\
&\le A(u) - \frac{3(p-1)(3p-5)}{4(p+1)} C(u) \\
&\quad + C e^2 
\left( \| \rho \|_{\frac{6}{5}} + \| x \cdot \nabla \rho \|_{\frac{6}{5}}
+ \| x \cdot (D^2 \rho x) \|_{\frac{6}{5}} \right) \| S_0(u) \|_6.
\end{align*}
Moreover since 
\begin{align*}
0 \ge Q(u) 
&= A(u) - \frac{3(p-1)}{2(p+1)} C(u) + e^2 D(u) -4 e^2 E_1(u) + e^2 E_2(u) \\
&\ge A(u) - \frac{3(p-1)}{2(p+1)} C(u) 
-C e^2 \left( \| \rho \|_{\frac{6}{5}} + \| x \cdot \nabla \rho \|_{\frac{6}{5}} \right) 
\| S_0(u) \|_6,
\end{align*}
it follows that
\[
\begin{aligned}
F''(\la) &\le - \frac{3(p-1)(3p-7)}{4(p+1)} C(u) 
+ C e^2 
\left( \| \rho \|_{\frac{6}{5}} + \| x \cdot \nabla \rho \|_{\frac{6}{5}} 
+ \| x \cdot (D^2 \rho x) \|_{\frac{6}{5}} \right) \| S_0(u) \|_6.
\end{aligned}
\]
Thus by Lemma \ref{lem:3.1} and Lemma \ref{lem:5.1}, we deduce that
\begin{align*}
F''(\la) 
&\le - \left\{ \frac{3(p-1)(3p-7)}{4(p+1)} 
- C e^2 
\big( \| \rho \|_{\frac{6}{5}} + \| x \cdot \nabla \rho \|_{\frac{6}{5}}
+ \| x \cdot (D^2 \rho x) \|_{\frac{6}{5}} \big) \right\} \| u \|_{p+1}^{p+1} \\
&< 0 \quad \text{for all} \ \la \ge 1.
\end{align*}
This ends the proof.
\end{proof}

By Lemma \ref{lem:5.6}, we obtain the following proposition,
which is also a fundamental tool for the strong instability.

\begin{proposition} \label{prop:5.7}
Suppose $\frac{7}{3} < p <5$ and let $u_0$ be the ground state solution of \ef{eq:1.1}.
There exists $\rho_0>0$ such that if 
$e^2 \left( \| \rho \|_{\frac{6}{5}} + \| x \cdot \nabla \rho \|_{\frac{6}{5}}
+ \| x \cdot (D^2 \rho x) \|_{\frac{6}{5}} \right) \le \rho_0$, it holds that
\begin{equation} \label{eq:5.12}
I(u_0^{\la}) < I(u_0), \quad Q(u_0^{\la}) < 0 
\quad \text{and} \quad J(u_0^{\la})<0 \quad \text{for all} \ \la >1.
\end{equation}

\end{proposition}

\begin{proof}
First by Lemma \ref{lem:5.6} and from $F'(1)=Q(u_0) =0$, we find that
\[
F'(\la) = F'(1) + \int_1^\la F''(\tau) \,d\tau <  0 \quad \text{for} \ \la >1,
\]
which yields that $F(\la) < F(1)$ and hence
$I(u_0^{\la}) < I(u_0)$ for all $\la >1$.

Next from \ef{eq:5.11}, one has
\[
G'(\la) = F'(\la) + \la F''(\la) < 0 \quad \text{for} \ \la >1
\]
from which we deduce that $Q(u_0^{\la}) < Q(u_0) =0$ for all $\la >1$.

Finally if $J(u_0^{\tilde{\la}})=0$ for some $\tilde{\la} >1$, 
it follows by Proposition \ref{prop:4.1} that $I(u_0) \le I(u_0^{\tilde{\la}})$.
This is a contradiction to the first assertion.
\end{proof}

We finish this section by mentioning that
we cannot expect to apply 
the classical approach due to \cite{BC} (see also \cite{FCW, LeC}).

Indeed in \cite{BC, FCW, LeC}, it was required to establish the variational
characterization of the ground state solutions as follows:
\begin{equation} \label{eq:5.13}
I(u_0) = \hat{\sigma} := \inf\{ I(u) \mid u \in H^1(\R^3,\C) \setminus \{ 0 \}, 
Q(u)=0 \}.
\end{equation}
If we could show that for any $u \in H^1(\R^3, \C)\setminus \{ 0 \}$ with $Q(u)=0$,
\begin{equation} \label{eq:5.14}
\text{there exists a unique $\hat{\la}>0$ such that $G(\hat{\la})= Q(u^{\hat{\la}})=0$},
\end{equation}
then we are able to prove \ef{eq:5.13}.
In fact since $Q(u)= \frac{3}{2}N(u)-P(u)$,
we have by Lemma \ref{lem:3.2} that $I(u_0) \ge \hat{\sigma}$.
On the other hand since $F'(\la) = \frac{G(\la)}{\la}$,
$G(\la) < 0$ for large $\la$ and $\hat{\la}=1$ for $Q(u)=0$, it follows that
\[
F'(\la) > 0 \ \text{for} \ 0 <\la <1 \quad \text{and} \quad
F'(\la) < 0 \ \text{for} \ \la >1
\]
and hence $F(\la) \le F(1)$ for any $\la >0$ and
$u \in H^1(\R^3, \C)\setminus \{ 0 \}$ with $Q(u)=0$.
Arguing as Lemma \ref{lem:5.2}, we also see that 
$J(u^{\la^*})=0$ for some $\la^*>0$.
Thus by Proposition \ref{prop:4.1}, we arrive at
\[
I(u_0) \le I(u^{\la^*}) \le I(u) \quad \text{for any} \ 
u \in H^1(\R^3, \C)\setminus \{ 0 \} \ \text{with} \ Q(u)=0,
\]
which yields that $I(u_0) \le \hat{\sigma}$ and hence \ef{eq:5.13}.

Our main concern is whether \ef{eq:5.14} holds or not.
When $Q(u)=0$ and $Q(u^{\la})=0$, we can see that
\[
0= \frac{3(p-1)}{2(p+1)} ( \la^2 - \la^{\frac{3(p-1)}{2}} ) C(u) 
+ e^2 (\la - \la^2) D(u) + \hat{R}(\la, u) \quad \text{for} \ \la >0,
\]
where
\begin{align*}
\hat{R}(\la,u) &:= e^2 \intR S_0(u) \hat{M}(\la, u ) \,dx, \\
\hat{M}(\la, x) &:= \frac{1}{\la^2} \left( \rho(\la^{-1}x) 
+ \frac{1}{2} (\la^{-1}x) \cdot \nabla \rho (\la^{-1} x) \right) 
- \la^2 \left( \rho(x) + \frac{1}{2} x \cdot \nabla \rho(x) \right).
\end{align*}
Thus if we could show that
\begin{equation} \label{eq:5.15}
\hat{R}(\la, u) > 0 \ \text{for} \ 0 < \la < 1
\quad \text{and} \quad \hat{R}(\la,u) < 0 \ \text{for} \ \la >1,
\end{equation}
we obtain the uniqueness of $\hat{t}$.
In order to confirm \ef{eq:5.15}, let us put
\[
A(\la, x) := \la^4 \left( \rho(\la x) + \frac{1}{2} (\la x) \cdot \nabla \rho( \la x) \right).
\]
Then we can see that
\[
\hat{M}(\la, x) = \la^2 \left( A(\la^{-1}x) - A(1,x) \right).
\]
Moreover \ef{eq:5.15} is satisfied if
\begin{equation} \label{eq:5.16}
\frac{\partial A}{\partial \la} (\la,x) > 0 \quad \text{for all} \ \la >0.
\end{equation}
By a direct calculation, one also has
\[
\frac{\partial A}{\partial \la}(\la,x)
= \frac{\la^3}{2} \left\{ 8 \rho(\la x) + 7 (\la x) \cdot \nabla \rho(\la x)
+ (\la x) \cdot \big( D^2 \rho(\la x) (\la x) \big) \right\}.
\]
Thus we are able to conclude that \ef{eq:5.16} holds by assuming
\begin{equation} \label{eq:5.17}
8 \rho(x) + 7 x \cdot \nabla \rho(x) + x \cdot ( D^2 \rho(x)x) >0
\quad \text{for all} \ x \in \R^3.
\end{equation}

However we cannot expect \ef{eq:5.17} to hold for doping profiles.
When $\rho$ is a Gaussian function $e^{-\alpha |x|^2}$ for $\alpha >0$,
it holds that
\[
8 \rho(x) + 7 x \cdot \nabla \rho(x) + x \cdot ( D^2 \rho(x)x) 
= 4 e^{-\alpha |x|^2} ( \alpha^2 |x|^4 - 4 \alpha |x|^2 + 2).
\]
Hence no matter how we choose $\alpha$,
\ef{eq:5.17} fails to hold for $\sqrt{ \frac{2- \sqrt{2}}{\alpha}} < |x| <
\sqrt{\frac{2+\sqrt{2}}{\alpha}}$. 
Moreover if we consider $\rho(x)= \frac{1}{1+|x|^{\alpha}}$, we find that
\[
\begin{aligned}
&8 \rho(x) + 7 x \cdot \nabla \rho(x) + x \cdot ( D^2 \rho(x)x) \\
&= \frac{1}{(1+|x|^{\alpha})^3}
\left( (\alpha-2)(\alpha-4) |x|^{2 \alpha} -(\alpha-2)(\alpha +8) |x|^{\alpha} + 8 \right).
\end{aligned}
\]
Then it folds that
\[
\begin{cases}
\alpha > 4 &\Rightarrow \quad 
\text{\ef{eq:5.17} fails to hold on some annulus}, \\
2 < \alpha \le 4 &\Rightarrow \quad
\text{\ef{eq:5.17} fails to hold for large $|x|$}, \\
\alpha < 2 &\Rightarrow \quad \rho \notin L^{\frac{6}{5}}(\R^3).
\end{cases}
\]
In this sense, the assumption (A1) and \ef{eq:5.17} seem to be inconsistent.

\section{Strong instability of standing waves}

In this section, we prove Theorem \ref{thm:1.1}.
For this purpose, let us introduce the set
\begin{equation} \label{eq:6.1}
\mathcal{B} := \{ u \in H^1(\R^3, \C) \mid
I(u) < I(u_0), \ J(u) < 0 , \ Q(u) <0 \},
\end{equation}
where $u_0$ is the ground state solution of \ef{eq:1.1}.
We recall that $\mathcal{I}(u)= I(u) + e^2 F$ and $F$ is independent of $u$,
meaning that we may replace $\mathcal{I}$ by $I$ in the definition of $\mathcal{B}$. 
Then we have the following lemma, which states that
the set $\mathcal{B}$ given in \ef{eq:6.1} is invariant under the flow for \ef{eq:2.1}.

\begin{lemma} \label{lem:6.1}
Suppose that $\frac{7}{3} < p < 5$ and assume $\psi_0 \in \mathcal{B}$.
Let $T^*$ be the maximal existence time for $\psi_0$.
There exists $\rho_0>0$ such that if 
$e^2 \left( \| \rho \|_{\frac{6}{5}} + \| x \cdot \nabla \rho \|_{\frac{6}{5}}
+ \| x \cdot (D^2 \rho x) \|_{\frac{6}{5}} \right) \le \rho_0$, then
the solution $\psi(t)$ of \ef{eq:2.1} with $\psi(0)=\psi_0$ belongs to 
$\mathcal{B}$ for all $t \in [0,T^*)$.
\end{lemma}

\begin{proof}
First by the energy conservation law and the charge conservation law
given in Proposition \ref{prop:2.1}, it follows that
\begin{equation} \label{eq:6.2}
I \big( \psi(t) \big)
= E \big( \psi(t) \big) + \frac{\omega}{2} \| \psi(t) \|_2^2 
= E (\psi_0 ) + \frac{\omega}{2} \| \psi_0 \|_2^2 
= I(\psi_0) < I(u_0) 
\end{equation}
for all $t \in [0,T^*)$.

Next if $J \big( \psi(t_1) \big) =0$ for some $t_1 \in (0,T^*)$, 
we have by Proposition \ref{prop:4.1} that
$I(u_0) \le I \big( \psi(t_1) \big)$, which contradicts \ef{eq:6.2}.
Thus by the continuity of the solution $\psi(t)$ with respect to $t$, 
we conclude that $J \big( \psi(t) \big) <0$ for all $t \in [0,T^*)$.

Finally suppose by contradiction that $Q \big( \psi(t_2) \big) =0$ 
for some $t_2 \in (0,T^*)$.
Then by Proposition \ref{prop:5.4}, it holds that
\[
\frac{1}{2} Q \big( \psi(t_2) \big) 
\le I \big( \psi(t_2) \big) - I(u_0).
\]
Again, this is a contradiction with \ef{eq:6.2}.
This completes the proof.
\end{proof}

Next we establish the following lemma.

\begin{lemma} \label{lem:6.2}
Suppose that $\frac{7}{3} < p < 5$ and let $u_0$ be the 
ground state solution of \ef{eq:1.1}.
There exists $\rho_0>0$ such that if 
$e^2 \left( \| \rho \|_{\frac{6}{5}} + \| x \cdot \nabla \rho \|_{\frac{6}{5}}
+ \| x \cdot (D^2 \rho x) \|_{\frac{6}{5}} \right) \le \rho_0$, then
$u_0^{\la} \in \mathcal{B}$ for all $\la >1$.
\end{lemma}

\begin{proof}
This is a direct consequence of Proposition \ref{prop:5.7}.
\end{proof}

Now we are ready to prove Theorem \ref{thm:1.1}.

\begin{proof}[Proof of Theorem \ref{thm:1.1}]
First by Lemma \ref{lem:4.3}, we find that $u_0 \in \Sigma$.
This together with Lemma \ref{lem:6.2} implies that
$u_0^{\la} \in \mathcal{B} \cap \Sigma$ for all $\la >1$.
Then by Lemma \ref{lem:6.1}, it follows that
the solution $\psi(t)$ of \ef{eq:2.1} with $\psi(0)=u_0^{\la}$ belongs to $\mathcal{B}$.
Moreover applying Lemma \ref{lem:2.2}, Proposition \ref{prop:5.4}
and using the conservation of $I$, we obtain
\[
V''(t) = 8 Q \big( \psi(t) \big) 
\le 16 \big( I \big( \psi(t) \big) - I(u_0) \big)
= 16 \big( I(u_0^{\la}) - I(u_0) \big) <0
\]
for all $t \in (0,T^*)$.
This implies that $T^*< \infty$ and the solution $\psi(t)$ blows up in finite time.

Finally since $u_0^{\la} \to u_0$ in $H^1(\R^3,\C)$ as $\la \searrow 1$,
we conclude that the standing wave $e^{i \omega t} u_0$ of \ef{eq:1.2}
is strongly unstable.
\end{proof}

\section{The case $\rho$ is a characteristic function}

In this section, we consider the case where the doping profile $\rho$
is a characteristic function, 
which appears frequently in physical literatures \cite{Je, MRS, Sel}.
More precisely, let $\{ \Omega_i \}_{i=1}^m \subset \R^3$ be
disjoint bounded open sets with smooth boundary.
For $\sigma_i>0$ $(i=1,\cdots, m)$, 
we assume that the doping profile $\rho$ has the form:
\begin{equation} \label{eq:7.1}
\rho(x)= \sum_{i=1}^m \sigma_i \chi_{\Omega_i}(x), \quad
\chi_{\Omega_i}(x)=
\begin{cases} 
1 & (x \in \Omega_i), \\
0 & (x \notin \Omega_i).
\end{cases}
\end{equation}
In this case, $\rho$ cannot be weakly differentiable so that the assumption
(A1) does not make sense.
Even so, we are able to obtain the strong instability of standing waves  
by imposing some smallness condition related with $\sigma_i$ and $\Omega_i$.

To state our main result for this case, 
let us put $\dis L:= \sup_{x \in \partial \Omega} |x| < \infty$.
A key point is the following \textit{sharp boundary trace inequality} 
due to \cite[Theorem 6.1]{A},
which we present here according to the form used in this paper.

\begin{proposition} \label{prop:7.1}
Let $\Omega \subset \R^3$ be a bounded domain with smooth boundary
and $\gamma: H^1(\Omega) \to L^2(\partial \Omega)$ be 
the trace operator.
Then it holds that
\[
\int_{\partial \Omega} | \gamma(u) |^2 \,dS
\le \kappa_1(\Omega) \int_{\Omega} |u|^2 \,dx
+ \kappa_2(\Omega) 
\left( \int_{\Omega} |u|^2 \,dx \right)^{\frac{1}{2}}
\left( \int_{\Omega} | \nabla u|^2 \,dx \right)^{\frac{1}{2}}
\]
for any $u \in H^1(\Omega)$,
where $\kappa_1(\Omega)= \frac{|\partial \Omega|}{|\Omega|}$,
$\kappa_2(\Omega) = \big\| | \nabla w| \big\|_{L^{\infty}(\partial\Omega)}$
and $w$ is a unique solution of the torsion problem:
\[
\Delta w = \kappa_1(\Omega) \ \hbox{in} \ \Omega, 
\quad \frac{\partial w}{\partial n} = 1 \ \hbox{on} \ \partial \Omega.
\]
\end{proposition}
In relation to the size of $\rho$, we define
\[
D(\Omega) := L | \Omega |^{\frac{1}{6}} 
\left( L \| H \|_{L^2(\partial \Omega)} + |\partial \Omega |^{\frac{1}{2}} \right)
\left( \kappa_1(\Omega) | \Omega |^{\frac{1}{3}} + \kappa_2 (\Omega) 
\right)^{\frac{1}{2}},
\]
where $H$ is the mean curvature of $\partial \Omega$.

\begin{remark} \label{rem:7.2}
It is known that $\kappa_2(\Omega) \ge 1$; see \cite{A}.
Then by the isoperimetric inequality in $\R^3$:
\[
| \partial \Omega | \ge 3 | \Omega |^{\frac{2}{3}} | B_1 |^{\frac{1}{3}},
\]
and by the fact that $| \Omega | \le |B_L(0)| = L^3 |B_1|$, we find 
\begin{equation} \label{eq:7.2}
D(\Omega) \ge
\left( \frac{|\Omega|}{|B_1|} \right)^{\frac{1}{3}} |\Omega|^{\frac{1}{6}}
\cdot \sqrt{3} |\Omega|^{\frac{1}{3}} |B_1|^{\frac{1}{6}}
\left( 3 |B_1|^{\frac{1}{3}} +1 \right)^{\frac{1}{2}} 
=C |\Omega|^{\frac{5}{6}} = C \| \chi_{\Omega} \|_{L^{\frac{6}{5}}},
\end{equation}
where $C$ is a positive constant independent of $\Omega$. 
\end{remark}

Under these preparations, we have the following result.

\begin{theorem} \label{thm:7.3}
Suppose that $\frac{7}{3}<p<5$ and assume that $\rho$ is given by \ef{eq:7.1}.
There exists $\rho_0>0$ such that if
\[
e^2 \sum_{i=1}^m \sigma_i D(\Omega_i) \le \rho_0,
\]
then the statement of Theorem \ref{thm:1.1} holds true.
\end{theorem}

\smallskip
We mention that the first place where $x \cdot \nabla \rho(x)$ 
and $x \cdot (D^2 \rho(x) x)$ appeared
was in the definition of $E_2(u)$ and $E_3(u)$ in \ef{eq:3.3}.
Under the assumption \ef{eq:7.1}, we replace them by
\begin{align*}
E_1(u) &= - \frac{1}{4} \sum_{i=1}^m \sigma_i 
\int_{\Omega_i} S_0(u) \,dx, \\
E_2(u) &:= - \frac{1}{2} \sum_{i=1}^m \sigma_i 
\int_{\partial \Omega_i} S_0(u) x \cdot n_i \,dS_i, \\
E_3(u) &:= - \frac{1}{2} \sum_{i=1}^m \sigma_i 
\int_{\partial \Omega_i} H_i(x) S_0(u) (x \cdot n_i)^2 \,dS_i, 
\end{align*}
where $n_i$ is the unit outward normal on $\partial \Omega_i$.
Then, we have the following.

\begin{lemma} \label{lem:7.4}
It holds that
\begin{align*}
\lim_{R \to \infty} \int_{B_R(0)} S_0(u) u x \cdot \nabla \bar{u} \,dx 
&= -10 E_1(u) + E_2(u), \\
\lim_{R \to \infty} \int_{B_R(0)} S_1(u) u x \cdot \nabla \bar{u} \,dx 
&= - 6 E_2(u) - E_3(u).
\end{align*}

\end{lemma}

\begin{proof}
For simplicity, let us consider the case $m=1$ and $\sigma=1$.
First by the divergence theorem and the fact $S_0(u)|u|^2 \in L^1(\R^3)$, 
one finds that
\[
\begin{aligned}
&\lim_{R \to \infty} \int_{B_R(0)} S_0(u) u x \cdot \nabla \bar{u} \,dx 
= -\frac{1}{8 \pi} \lim_{R \to \infty} \int_{\Omega} \int_{B_R(0)} 
\frac{u(y) y \cdot \nabla \overline{u(y)}}{|x-y|} \,d y \,d x \\
&= \frac{1}{16 \pi} \int_{\Omega} \int_{\mathbb{R}^3} 
\frac{|u(y)|^2 \operatorname{div}_y y}{|x-y|} d y d x
+\frac{1}{16 \pi} \int_{\Omega} \int_{\mathbb{R}^3} 
|u(y)|^2 y \cdot \nabla_y\left(\frac{1}{|x-y|}\right) d y dx \\
&= \frac{3}{16 \pi} \int_{\Omega} \int_{\mathbb{R}^3} 
\frac{|u(y)|^2}{|x-y|} \,d y \,d x
+\frac{1}{16 \pi} \int_{\Omega} \int_{\mathbb{R}^3}
|u(y)|^2 \frac{y \cdot(x-y)}{|x-y|^3} \,d y \,d x. 
\end{aligned}
\]
Using the identity $y \cdot (x-y) = -|x-y|^2+x \cdot (x-y)$,
the Fubini theorem and the divergence theorem, we get
\[
\begin{aligned}
&\lim_{R \to \infty} \int_{B_R(0)} S_0(u) u x \cdot \nabla \bar{u} \,dx \\
&= \frac{1}{8 \pi} \int_{\Omega} \int_{\mathbb{R}^3} 
\frac{|u(y)|^2}{|x-y|} \,d y \,d x
-\frac{1}{16 \pi} \int_{\mathbb{R}^3} \int_{\Omega}
|u(y)|^2 x \cdot \nabla_x \left(\frac{1}{|x-y|}\right) \,d x \,d y \\
&= \frac{1}{8 \pi} \int_{\Omega} \int_{\R^3} 
\frac{|u(y)|^2}{|x-y|} \,dy \,dx 
-\frac{1}{16 \pi} \int_{\mathbb{R}^3} \int_{\Omega} 
\operatorname{div}_x \left( \frac{|u(y)|^2 x}{|x-y|}\right) \,d x \,d y 
+\frac{3}{16 \pi} \int_{\R^3} \int_{\Omega} \frac{|u(y)|^2}{|x-y|} \,d x \,d y \\
&= \frac{5}{16 \pi} \int_{\Omega} \int_{\mathbb{R}^3} 
\frac{|u(y)|^2}{|x-y|} \,d y \,d x
-\frac{1}{16 \pi} \int_{\mathbb{R}^3} \int_{\Omega} 
\frac{|u(y)|^2}{|x-y|} x \cdot n \,d S \,d y \\
&= \frac{5}{2} \int_{\Omega} S_0(u) \,d x
-\frac{1}{2} \int_{\partial \Omega} S_0(u) x \cdot n \,d S
=-10 E_1(u)+ E_2(u).
\end{aligned}
\]
Similarly, we have
\begin{align} \label{eq:7.3}
&\lim_{R \to \infty} \int_{B_R(0)} S_1(u) u x \cdot \nabla \bar{u} \,dx 
= -\frac{1}{8 \pi} \lim_{R \to \infty} \int_{\partial \Omega} \int_{B_R(0)} 
\frac{u(y) y \cdot \nabla \overline{u(y)}}{|x-y|} x \cdot n \,dy \,dS \notag \\
&= \frac{1}{8 \pi} \int_{\partial \Omega} \int_{\mathbb{R}^3} 
\frac{|u(y)|^2 x \cdot n}{|x-y|} \,dy \,dS 
-\frac{1}{16 \pi} \intR \int_{\partial \Omega} 
|u(y)|^2 x \cdot \nabla_x \left( \frac{1}{|x-y|} \right) x \cdot n \,dS \,dy \notag \\
&= \frac{1}{8 \pi} \int_{\partial \Omega} \int_{\mathbb{R}^3} 
\frac{|u(y)|^2 x \cdot n}{|x-y|} \,dy \,dS
-\frac{1}{16 \pi} \intR \int_{\partial \Omega}
\operatorname{div}_x 
\left( \frac{|u(y)|^2 (x \cdot n) x}{|x-y|} \right) \,dS \,dy \notag \\
&\quad +\frac{3}{16 \pi} \intR \int_{\partial \Omega} 
\frac{|u(y)|^2 x \cdot n}{|x-y|} \,dS \,dy
+\frac{1}{16 \pi} \intR \int_{\partial \Omega} 
\frac{|u(y)|^2 x \cdot \nabla_x (x \cdot n)}{|x-y|} \,dS \,dy \notag \\
&= \frac{5}{2} \int_{\partial \Omega} S_0(u) x \cdot n \,d S
-\frac{1}{2} \int_{\partial \Omega} 
\operatorname{div}_x 
\big( S_0(u) (x \cdot n) x \big) \,dS \notag \\
&\quad +\frac{1}{2} \int_{\partial \Omega} 
S_0(u) x \cdot \nabla_x (x \cdot n) \,dS. 
\end{align}
Applying the surface divergence theorem (see e.g. \cite[7.6]{Si})
and noticing that $\partial (\partial \Omega)= \emptyset$, it follows that
\[
\begin{aligned}
\int_{\partial \Omega} \operatorname{div}_x (S_0(u) (x \cdot n) x) \,d S 
& =-\int_{\partial \Omega} \big( S_0(u)(x \cdot n) x \big)^{\perp} \cdot \vec{H} \,d S 
=-\int_{\partial \Omega} H(x)S_0(u) (x \cdot n)^2 \,d S,
\end{aligned}
\]
where $x^{\perp}$ is the normal component of $x$ and 
$\vec{H}$ is the mean curvature vector $\vec{H}= H n$.
Finally since $x \cdot \nabla_x (x \cdot n)= x \cdot n$, 
we deduce from \ef{eq:7.3} that
\[
\begin{aligned}
&\lim_{R \to \infty} \int_{B_R(0)} S_1(u) u x \cdot \nabla \bar{u} \,dx \\
&= 3 \int_{\partial \Omega} S_0(u) x \cdot n \,d S
+\frac{1}{2} \int_{\partial \Omega} H(x) S_0(u) (x \cdot n)^2 \,d S =-6 E_2(u)-E_3(u),
\end{aligned}
\]
which ends the proof for $m=1$ and $\sigma=1$.
The general case can be shown by summing up the integrals.
\end{proof}

By Lemma \ref{lem:7.4}, the Nehari-Pohozev functional
and the functional associated with the virial identity can be reformulated as follows.

\begin{lemma} \label{lem:7.5}
Under the assumption \ef{eq:7.1}, 
the functionals $Q(u)$ is still given by \ef{eq:3.11} and $J(u)$ by \ef{eq:4.1}.
\end{lemma}

\begin{proof}
It is known that the Pohozaev identity $P(u)=0$ 
can be obtained by multiplying \ef{eq:1.1} by $x \cdot \nabla \bar{u}$
and $ex \cdot S_0(u)$, integrating the resulting equations over $B_R(0)$ and passing to the limit $R \to \infty$.
Then, using Lemma \ref{lem:7.4}, one can prove that \ef{eq:3.9} and \ef{eq:3.10} are still valid.
Since $Q(u)= \frac{3}{2}N(u)-P(u)$ and $J(u)=2N(u)-P(u)$,
we conclude that $Q(u)$ and $J(u)$ are given respectively by \ef{eq:3.11} and \ef{eq:4.1}.
\end{proof}

Next we establish estimates for $E_1$, $E_2$ and $E_3$.

\begin{lemma} \label{lem:7.6}
For any $u \in H^1(\R^3,\C)$, $E_1$, $E_2$ and $E_3$ satisfy the estimates:
\[
\begin{aligned}
|E_1(u)| &\le C \sum_{i=1}^m \sigma_i |\Omega_i|^{\frac{5}{6}}
\| u \|^2, \\
|E_2(u)| &\le C \sum_{i=1}^m \sigma_i D(\Omega_i)
\| u \|^2, \\
|E_3(u)| &\le C \sum_{i=1}^m \sigma_i D(\Omega_i)
\| u \|^2,
\end{aligned}
\]
where $C>0$ is a constant independent of $\Omega_i$.
\end{lemma}

\begin{proof}
First we observe that 
\[
|E_1(u)| \le \frac{1}{4} \sum_{i=1}^m \sigma_i \int_{\Omega_i} |S_0(u)|\,dx
\le \frac{1}{4} \sum_{i=1}^m \sigma_i 
\left( \int_{\Omega_i} |S_0(u)|^6 \,dx \right)^{\frac{1}{6}}
\left( \int_{\Omega_i} \,dx \right)^{\frac{5}{6}},
\]
from which the estimate for $E_1$ can be obtained by Lemma \ref{lem:3.1}. 
Next by Lemma \ref{lem:3.1}, Proposition \ref{prop:7.1}, 
the H\"older inequality and the Sobolev inequality, one has
\begin{align*}
&|E_2(u)| \le \frac{1}{2} \sum_{i=1}^m \sigma_i
\int_{\partial \Omega_i} |S_0(u)| |x| \,dS_i \\
&\le \frac{1}{2} \sum_{i=1}^m \sigma_i 
\left( \int_{\partial \Omega_i} |S_0(u)|^2 \,dS_i \right)^{\frac{1}{2}}
\left( \int_{\partial \Omega_i} |x|^2 \,dS_i \right)^{\frac{1}{2}} \\
&\le \frac{1}{2} \sum_{i=1}^m \sigma_i L_i | \partial \Omega_i|^{\frac{1}{2}}
\Big( \kappa_1(\Omega_i) \| S_0(u) \|_{L^2(\Omega_i)}^2 
+\kappa_2(\Omega_i) \| S_0(u) \|_{L^2(\Omega_i)}
\| \nabla S_0(u) \|_{L^2(\Omega_i)} \Big)^{\frac{1}{2}} \\
\hspace{-1em} &\le \frac{1}{2} \sum_{i=1}^m \sigma_i L_i |\partial \Omega_i|^{\frac{1}{2}}
\Big( \kappa_1(\Omega_i) |\Omega_i|^{\frac{2}{3}} \| S_0(u) \|_{L^6(\R^3)}^2 
 +\kappa_2(\Omega_i) | \Omega_i |^{\frac{1}{3}} \| S_0(u) \|_{L^6(\R^3)}
\| \nabla S_0(u) \|_{L^2(\R^3)} \Big)^{\frac{1}{2}} \\
&\le C \sum_{i=1}^m \sigma_i L_i |\Omega_i|^{\frac{1}{6}}
|\partial \Omega_i|^{\frac{1}{2}} 
\left( \kappa_1(\Omega_i)|\Omega_i|^{\frac{1}{3}} 
+\kappa_2(\Omega_i) \right)^{\frac{1}{2}} 
\| \nabla S_0(u) \|_{L^2(\R^3)} \\
&\le C \sum_{i=1}^m \sigma_i D(\Omega_i) 
\| u \|^2.
\end{align*}
Similarly, we obtain
\[
\begin{aligned}
|E_3(u)| &\le \frac{1}{2} \sum_{i=1}^m \sigma_i \int_{\partial \Omega_i}
|H_i| |S_0(u)| |x|^2 \,d S_i \\
&\le \frac{1}{2} \sum_{i=1}^m \sigma_i L_i^2 \|H_i\|_{L^2(\partial \Omega_i)}
\|S_0(u)\|_{L^2(\partial \Omega_i)} \\
&\le C \sum_{i=1}^m \sigma_i L_i^2 \| H_i \|_{L^2(\partial \Omega_i)}
| \Omega_i |^{\frac{1}{6}} \left( \kappa_1(\Omega_i)|\Omega_i|^{\frac{1}{3}} 
+\kappa_2(\Omega_i) \right)^{\frac{1}{2}} 
\| \nabla S_0(u) \|_{L^2(\R^3)} \\
&\le C \sum_{i=1} \sigma_i D(\Omega_i) \| u \|^2.
\end{aligned}
\]
This completes the proof.
\end{proof}

Our next step is to modify the proof of the energy inequality in Lemma \ref{lem:5.3}.
For this purpose, we prove the following.

\begin{lemma} \label{lem:7.7}
Let $\Omega \subset \R^3$ be a bounded domain with smooth boundary
and put
\[
\Omega(\la) := \int_{\la \Omega} S_0(u)(x) \,dx
= \la^3 \int_{\Omega} S_0(u)(\la y) \,dy.
\]
Then it holds that
\[
\begin{aligned}
\Omega'(\la) &= \la^2 \int_{\partial \Omega} S_0(u)(\la y) (y \cdot n) \,dS, \\
\Omega''(\la) &= -2\la \int_{\partial \Omega} S_0(u)(\la y) (y \cdot n) \,dS
- \la \int_{\partial \Omega} H(y) S_0(u)(\la y) (y \cdot n)^2 \,dS.
\end{aligned}
\]

\end{lemma}

\begin{proof}
First we observe that 
\[
\begin{aligned}
\Omega^{\prime}(\la)
&=3 \la^2 \int_{\Omega} S_0(u)(\la y) \,d y
+ \la^2 \int_{\Omega} \nabla_y S_0(u)(\la y) \cdot y \,d y \\
&= 3 \la^2 \int_{\Omega} S_0(u)(\la y) \,dy
+ \la^2 \int_{\Omega} \operatorname{div}_y \big( S_0(u)(\la y) y \big) \,dy
- \la^2 \int_{\Omega} S_0(u)(\la y) \operatorname{div}_y y \,dy \\
&= \la^2 \int_{\partial \Omega} S_0(u)(\la y) (y \cdot n) \,dS.
\end{aligned}
\]
Similarly by the surface divergence theorem, one has
\[
\begin{aligned}
&\Omega^{\prime \prime}(\la) 
= 2 \la \int_{\partial \Omega} S_0(u)(\la y) (y \cdot n) \,d S 
+ \la \int_{\partial \Omega} \nabla_y S_0(u)(\la y) \cdot y (y \cdot n) \,d S \\
&= 2 \la \int_{\partial \Omega} S_0(u)(\la y) (y \cdot n) \,dS \\
&\quad + \la \int_{\partial \Omega} 
\operatorname{div}_y \big( S_0(u)(\la y) (y \cdot n)y \big) \,dS 
- \la \int_{\partial \Omega} 
S_0(u)(\la y) \operatorname{div}_y \big( (y \cdot n) y \big) \,dS \\
&= 2 \la \int_{\partial \Omega} S_0(u)(\la y) (y \cdot n) \,dS
- \la \int_{\partial \Omega} 
H(y) S_0(u)(\la y) (y \cdot n)^2 \,dS \\
&\quad - \la \int_{\partial \Omega} 
S_0(u)(\la y) \operatorname{div}_y y (y \cdot n) \,dS 
- \la \int_{\partial \Omega} 
S_0(u)(\la y) y \cdot \nabla_y (y \cdot n) \,dS \\
&= -2 \la \int_{\partial \Omega} S_0(u)(\la y) (y \cdot n) \,d S
- \la \int_{\partial \Omega} H(y) S_0(u)(\la y) (y \cdot n)^2 \,d S.
\end{aligned}
\]
This completes the proof.
\end{proof}

\begin{remark} \label{rem:7.8}
Lemma \ref{lem:7.7} is related with the
''calculus of moving surfaces'' due to Hadamard;
see \cite[(38)-(39)]{Gri}.
\end{remark}

Using Lemma \ref{lem:7.6} and Lemma \ref{lem:7.7}, 
we can establish the energy identity \ef{eq:5.5} as follows.

\begin{lemma} \label{lem:7.9}
Suppose that $\frac{7}{3}<p<5$.
Under the assumption \ef{eq:7.1}, 
there exist $C_1$, $C_2>0$ such that
the following estimate holds: For any $u \in H^1(\R^3,\C)$ and 
$\la \in [\delta^*,1]$,
\[
f(\la) - f(1) 
\le - C_1 (1-\la)^2 \| u \|_{p+1}^{p+1}
+ C_2(1-\la)^2 e^2 \sum_{i=1}^m \sigma_i D(\Omega_i) \| u \|^2.
\]
\end{lemma}

\begin{proof}
Under the notation of Lemma \ref{lem:7.7}, 
we can write the remainder term $R(\la,u)$ as follows:
\[
\begin{aligned}
R(\la, u) &= 2e^2 E_1(u^{\la}) 
+2 e^2(\la^2-2) E_1(u) -\frac{e^2}{2} (\la^2-1) E_2(u) \\
&= e^2 \sum_{i=1}^m \sigma_i \Big\{
- \frac{\la^{-2}}{2} \int_{\la \Omega_i} S_0(u) \,dx
- \frac{\la^2-2}{2} \int_{\Omega_i} S_0(u) \,dx 
+ \frac{\la^2-1}{4} \int_{\partial \Omega_i} S_0(u) x \cdot n_i \,dS_i \Big\} \\
&= - \frac{e^2}{4} \sum_{i=1}^m \sigma_i \left\{
2 \la^{-2} \Omega_i(\la) + 2(\la^2 -2) \Omega_i(1) - (\la^2-1) \Omega_i'(1) \right\}.
\end{aligned}
\]
Let us put
\[
H(\la) := 2 \la^{-2} \Omega_i(\la) + 2(\la^2 -2) \Omega_i(1) - (\la^2-1) \Omega_i'(1).
\]
Then one finds that $H(1)=H'(1)=0$ and
\[
\begin{aligned}
H^{\prime \prime}(\la) 
&= 12 \la^{-4} \Omega_i(\la) - 8 \la^{-3} \Omega_i'(\la) 
+2 \la^{-2} \Omega_i''(\la) 
+ 4 \Omega_i(1) - 2 \Omega_i'(1) \\
&\ge -\frac{1}{\la^4} 
\left( 12 |\Omega_i(\la)| + 8 \la |\Omega_i'(\la)| + 2\la^2 |\Omega_i''(\la)| \right)
-4 |\Omega_i(1)| - 2 |\Omega_i'(1)| 
=: -\tilde{N}(\la).
\end{aligned}
\]
Thus by the Taylor theorem, 
there exists $\xi = \xi(\la) \in \left( \delta^*, 1 \right)$ such that
\[
H(\la) \ge -\frac{1}{2} \tilde{N}(\xi)(1-\la)^2 \quad \text{for} \ 
\delta^* \le \la \le 1.
\]

Now by Proposition \ref{prop:7.1} and Lemma \ref{lem:7.7},
arguing similarly as Lemma \ref{lem:7.6}, one finds that
\[
\begin{aligned}
\frac{|\Omega_i(\la)|}{\la^3}
&\le \int_{\Omega_i} | S_0(u)(\la y)| \,d y 
\le |\Omega_i|^{\frac{5}{6}} \| S_0(u) (\la y) \|_{L^6(\R^3)} 
\le C \la^{-\frac{1}{2}} D(\Omega_i) \| u \|^2, \\
\frac{|\Omega_i^{\prime}(\la)|}{\la^2}
&\le \int_{\partial \Omega_i} | S_0(u)(\la y)| |y| \,dS 
\le C \la^{-\frac{1}{2}} D(\Omega_i) \| u \|^2, \\
\frac{|\Omega_i^{\prime \prime}(\la)|}{\la}
&\le 2 \int_{\partial \Omega_i} | S_0(u)(\la y)| |y| \,dS
+ \int_{\partial \Omega_i} |H_i(y)| | S_0(u)(\la y)| |y|^2 \,dS 
\le C \la^{-\frac{1}{2}} D(\Omega_i) \| u \|^2.
\end{aligned}
\]
Thus we have
\[
\tilde{N}(\xi) \le \frac{C}{(\delta^*)^{\frac{3}{2}}} D(\Omega_i) \| u \|^2
+C D(\Omega_i) \| u \|^2.
\]
and hence
\[
R(\la, u) \ge - C_2(1- \la)^2 e^2 \sum_{i=1}^m \sigma_i
D(\Omega_i) \| u \|^2
\]
for some $C_2>0$ independent of $e$, $\rho$, $\la$.
The remaining parts can be shown 
in the same way as Lemma \ref{lem:5.3}.
\end{proof}

\begin{proof}[Proof of Theorem \ref{thm:7.3}]
By Lemma \ref{lem:7.5}, Lemma \ref{lem:7.6} and Lemma \ref{lem:7.9}, 
we are able to modify the proofs in Sections 2-6. 
\end{proof}

\bigskip
\noindent {\bf Acknowledgements.}

The authors are grateful to anonymous referees for 
carefully reading the manuscript and providing us valuable comments.

The second author has been supported by JSPS KAKENHI Grant Numbers 
JP21K03317, JP24K06804.
This paper was carried out while the second author was staying 
at the University of Bordeaux. 
The second author is very grateful to all the staff of the University of Bordeaux 
for their kind hospitality.


\end{document}